\documentclass[letterpaper,12pt]{article}
\usepackage{amsmath}
\usepackage{amsfonts}
\usepackage{amssymb}
\usepackage{amsthm}
\usepackage{hyperref}
\usepackage{enumerate}
\usepackage[margin=1in]{geometry}
\usepackage[all,cmtip]{xy}

\newtheorem{theorem}{Theorem}[section]
\newtheorem{rem}[theorem]{Remark}
\newenvironment{remark}{\begin{rem}\rm}{\end{rem}}

\newtheorem{proposition}[theorem]{Proposition}

\newtheorem{lemma}[theorem]{Lemma}
\newtheorem{corollary}[theorem]{Corollary}
\newtheorem{eg}[theorem]{Example}
\newenvironment{example}{\begin{eg}\rm}{\end{eg}}
\newtheorem{definition}[theorem]{Definition}

\newcommand{\C}{\mathbb{C}}
\newcommand{\Cl}{\mathbb{C}\mathrm{l}}
\newcommand{\Z}{\mathbb{Z}}
\newcommand{\R}{\mathbb{R}}
\newcommand{\V}{\mathcal{V}}
\newcommand{\W}{\mathcal{W}}
\renewcommand{\S}{\mathcal{S}}
\newcommand{\D}{\mathsf{D}}
\newcommand{\A}{\mathcal{A}}
\newcommand{\dbbar}{\overline{\partial}_b}
\newcommand{\jex}{\mathcal{J}(E,X)}

\newcommand{\g}{\mathfrak{g}}
\newcommand{\hatA}{\hat{\mathrm{A}}}
\renewcommand{\div}{\mathrm{div}}
\renewcommand{\phi}{\varphi}

\DeclareMathOperator{\Span}{span}
\DeclareMathOperator{\End}{End}
\DeclareMathOperator{\Img}{im}
\DeclareMathOperator{\Id}{Id}
\DeclareMathOperator{\rank}{rank}

\DeclareMathOperator{\ind}{index}
\DeclareMathOperator{\Ch}{Ch}
\DeclareMathOperator{\Td}{Td}

\DeclareMathOperator{\Str}{Str}
\DeclareMathOperator{\Tr}{Tr}

\DeclareMathOperator{\curv}{curv}
\DeclareMathOperator{\charr}{char}
\title{On transversally elliptic operators and the quantization of manifolds with $f$-structure}
\author{Sean Fitzpatrick\footnote{Research supported by an NSERC postdoctoral fellowship}\\
Department of Mathematics\\
University of California, Berkeley}

\begin{document}

\maketitle

\begin{abstract}
An $f$-structure on a manifold $M$ is an endomorphism field $\varphi\in\Gamma(M,\End(TM))$ such that $\varphi^3+\varphi=0$.  Any $f$-structure $\varphi$ determines an almost CR structure $E_{1,0}\subset T_\C M$ given by the $+i$-eigenbundle of $\varphi$.  Using a compatible metric $g$ and connection $\nabla$ on $M$, we construct an odd first-order differential operator $\D$, acting on sections of $\S=\Lambda E_{0,1}^*$, whose principal symbol is of the type considered in \cite{F2}.  In the special case of a CR-integrable almost $\S$-structure, we show that when $\nabla$ is the generalized Tanaka-Webster connection of Lotta and Pastore, the operator $\D$ is given by $\D = \sqrt{2}(\dbbar+\dbbar^*)$, where $\dbbar$ is the tangential Cauchy-Riemann operator.

We then describe two ``quantizations'' of manifolds with $f$-structure that reduce to familiar methods in symplectic geometry in the case that $\varphi$ is a compatible almost complex structure, and to the contact quantization defined in \cite{F4} when $\varphi$ comes from a contact metric structure.  The first is an index-theoretic approach involving the operator $\D$; for certain group actions $\D$ will be transversally elliptic, and using the results in \cite{F2}, we can give a Riemann-Roch type formula for its index.  The second approach uses an analogue of the polarized sections of a prequantum line bundle, with a CR structure playing the role of a complex polarization.
\end{abstract}
\section{Introduction}
Let $M$ be a smooth, compact manifold, and let $\varphi\in\Gamma(M,\End(TM))$ be an $f$-structure on $M$; that is, $\varphi$ is an endomorphism field satisfying
\begin{equation*}
 \varphi^3+\varphi = 0.
\end{equation*}
Such structures were introduced by Yano \cite{Yano}.
The complementary projection operators $l=-\varphi^2$ and $m=\varphi^2+\Id_{TM}$ determine a splitting $TM = \ker\varphi\oplus\Img\varphi$ of the tangent bundle of $M$. The restriction of $\varphi$ to $\Img \varphi$ squares to $-\Id_{\Img\varphi}$, and thus, as noted in \cite{LP}, an $f$-structure is equivalent to an almost CR structure together with a choice of complement to the Levi distribution.

In \cite{F2}, we used almost CR structures to construct new examples of transversally elliptic symbols (in the sense of Atiyah \cite{AT}), and gave a formula for their (cohomological) equivariant index.  In this paper we will give a construction of a first-order differential operator whose principal symbol is of the type considered in \cite{F2}.  Such an operator was introduced in the contact setting in \cite{F1}, and the general approach first appeared in the author's thesis \cite{F3}.  In \cite{F2} we required the existence of a subbundle $E\subset TM$ of constant rank, and a group action on $M$ such that the orbits of $G$ are transverse to the subbundle $E$, in a sense we will make precise.  While this construction does not produce the most general transversally elliptic operators, it does include many of the best-known examples of transversally elliptic operators (or symbols) encountered, for example, in \cite{AT,BV2,V1}.    

Given a manifold $M$ with $f$-structure $\varphi$, it is always possible to find a compatible metric $g$ and connection $\nabla$ \cite{Soare} satisfying
\[
g(\phi X,Y) + g(X,\phi Y) = 0 \quad \text{and} \quad \nabla\phi=\nabla g = 0.
\]
The eigenvalues of $\varphi$ (acting on $T_\C M :=TM\otimes\C$) are 0 and $\pm i$; we let $E=\Img \varphi$, and let $E_{1,0}\subset T_\C M$ denote the $+i$-eigenbundle of $\varphi$ which, as noted above, defines an almost CR structure on $M$.  We use the data $(\varphi, g,\nabla)$  to construct an odd first-order differential operator $\D$ acting on sections of $\S=\Lambda E^*_{0,1}$, where $E_{0,1} = \overline{E_{1,0}}$.  The construction is based on the usual construction of a Dirac operator on an almost Hermitian manifold (see \cite{BGV,Nic}): the metric $g$ allows us to construct the bundle of Clifford algebras $\Cl(E)$, whose fibre over $x\in M$ is the complexified Clifford algebra of $E^*_x$ with respect to the inner product induced by $g$.  The Clifford bundle then acts on $\S$ via the Clifford action $\mathbf{c}$ defined for $\alpha\in\Gamma(M,E^*)$ by $\mathbf{c}(\alpha)\gamma = \sqrt{2}\left(\alpha^{0,1}\wedge\gamma - \iota(\alpha^{1,0})\gamma\right)$, where the contraction is defined in terms of $g$.  The bundles $E^*$ and $T^*$ are orthogonal with respect to the metric $g$, and we let $\pi_{E^*}:T^*M\to E^*$ denote the orthogonal projection. We then define $\D$ by the composition
\[
\Gamma(M,\S)\xrightarrow[]{\nabla}\Gamma(M,T^*M\otimes\S)\xrightarrow[]{\pi_{E^*}}\Gamma(M,E^*\otimes\S)\xrightarrow[]{\mathbf{c}}\Gamma(M,\S).
\]
The principal symbol of $\D$ is given by $\sigma_P(\D)(x,\zeta) = i\mathbf{c}(\pi_{E^*}(\zeta_x))$ for $(x,\zeta)\in T^*M$, so that the results of \cite{F2} apply to the operator $\D$.

When the almost CR structure determined by $\varphi$ is CR-integrable, we can also define the $\dbbar$ operator of the resulting tangential CR complex, and construct another odd first-order differential operator acting on sections of $\S$, given in this case by $\D_b = \sqrt{2}(\dbbar+\dbbar^*)$, where $\dbbar^*$ denotes the formal adjoint of $\dbbar$, defined using the metric $g$.  This operator satisfies $\D_b^2 = 2\Box_b$, where $\Box_b$ denotes the Kohn-Rossi Laplacian \cite{KR} (see also \cite{FS}).  When $M$ is equipped with the additional structure of an {\em almost $\S$-manifold}, as defined in \cite{DIP}, Lotta and Pastore have shown \cite{LP} that there exists a canonical connection $\nabla^{LP}$ analogous to the Tanaka-Webster connection of a strongly pseudoconvex CR manifold of hypersurface type.  One of our main results is a proof that if we take $\nabla=\nabla^{LP}$ in the definition of $\D$ given above, then  $\D=\D_b$.

We then consider the case of a compact Lie group $G$ acting smoothly on $M$ such that $\varphi$, $g$ and $\nabla$ (and hence $\D$) are $G$-invariant.  Such group action preserves the splitting $T_\C M = E_{1,0}\oplus E_{0,1}\oplus (T\otimes\C)$, where $T=\ker\varphi$.  If $T$ is contained in the span of the vector fields generated by the infinitesimal action of $\mathfrak{g}$, then the operator $\D$ is $G$-transversally elliptic, since $\sigma_P(\D)$ is invertible for all nonzero $\alpha\in E^*_x$.  The equivariant index of $\D$ can therefore be defined as a distribution (i.e. generalized function) on $G$ \cite{AT}.  We can also twist this construction by an equivariant Hermitian vector bundle $\V$, and extend $\D$ to an operator $\D_\V$ acting on $\V$-valued differential forms.  In the almost $\S$ case, where the subbundle $T$ is trivial, the germ of the equivariant index of $\D$ near the identity element in $G$ is given (for $X\in\g$ sufficiently small) by the formula
\[
 \ind^G(\D_\V)(e^X) = \frac{1}{(2\pi i)^n}\int_M \Td(E,X)\Ch(\V,X)\jex,
\]
where $E\otimes\C = E_{1,0}\oplus E_{0,1}$, $n=\rank E/2$, and $\jex$ is an equivariant differential form with generalized coefficients defined as follows: Let $\theta\in\mathcal{A}^1(T^*M)$ denote the Liouville 1-form on $T^*M$, let $\iota:E^0\hookrightarrow T^*M$ denote the inclusion of the annihilator of $E$ (which we identify with $T^*$), and let $p:E^0\to M$ denote the projection mapping.  The equivariant differential of $\theta$ is given by $D\theta(X) = d\theta - \theta(X_M)$, where $X_M$ denotes the vector field generated by $X\in\g$, and $\jex$ is defined by
\[
 \jex = (2\pi i)^{-\rank T}p_*\iota^*e^{iD\theta(X)}.
\]
The general formula for the equivariant index near other elements of $G$, and for the case when $T$ is not trivial, is given similarly by
\[
 \ind^G(\mathsf{D})(ge^X) = \int_{M(g)}(2\pi i)^{-\rank E(g)/2}\Theta_g(X)\mathcal{J}(E(g),X),
\]
for $X\in\mathfrak{g}(g)$ sufficiently small, where
\[
 \Theta_g(X) = \frac{\Td(E(g),X)}{D_g^\C(\mathcal{N}_E,X)}\frac{\hatA^2(T(g),X)}{D_g(\mathcal{N}_T,X)}\Ch_g(\V,X).
\]
  An interesting special case is when $M=G/H$ is a complex homogeneous space, in which case the general index formula given above gives the character of a $G$-representation induced from a given $H$-representation $V$ (where we twist by $\V = G\times_H V$).  There are two obvious choices of $f$-structure on $G/H$; if we take the $f$-structure given by the complex structure, we obtain the holomorphic induced representation.  At the other extreme, we can take the $f$-structure $\phi=0$, in which case the character is that of the $L^2$ induced representation, by a result of Berline and Vergne \cite{BV0}.  (There are $f$-structures of intermediate rank as well; see \cite{F2}.)

In the final section of this article, we describe two ways of constructing a Hilbert space associated to a given $f$-structure.  We refer to these constructions as ``quantizations'' since special cases include well-known versions of the geometric quantization of a symplectic manifold (when $\phi$ is a compatible almost complex structure), as well as the two quantizations of contact manifolds defined in \cite{F4}.  In one approach, we define the quantization $Q(M)$ to be the $\Z_2$-graded Hilbert space $Q(M) = \ker \D\oplus\ker \D^*$.  Given a $G$-action on $M$ such that $\D$ is $G$-invariant, $Q(M)$ becomes a virtual $G$-representation, and when $\D$ is $G$-transversally elliptic, this representation has a well-defined (distributional) virtual character given by the equivariant index of $\D$.  When the fundamental 2-form $\Phi$ given by 
\[
 \Phi(X,Y) = g(\varphi X, Y)
\]
is closed, it defines a symplectic structure on the fibres of $E$. If there exists a Hermitian line bundle $\mathbb{L}$ equipped with a connection $\nabla^\mathbb{L}$ whose curvature form is equal to $i\Phi$, then $\mathbb{L}$ is a {\em quantum bundle} in the sense of \cite{DT}.  Using such a bundle, we can also produce an analogue of the Kostant-Souriau approach to geometric quantization, by defining a prequantization in the usual way.  If our $f$-structure is CR-integrable, the resulting CR structure is a natural analogue of a complex polarization; if in addition $\mathbb{L}$ is {\em CR holomorphic}, we can identify the space of polarized sections with the CR holomorphic $L^2$ sections of $\mathbb{L}$, and take this to be an alternative definition of $Q(M)$. 
Moreover, the Kostant algebra $\mathcal{K}(M,\Phi) = C^\infty(M)\times\Gamma(M,TM)$, which is equipped with the bracket \cite{Kost, Vais1}
\[
 [(f,X),(g,Y) = (X\cdot g-Y\cdot f + \Phi(X,Y),[X,Y]),
\]
has the representation on $Q(M)$ by the skew-Hermitian operators
\[
(f,X)\mapsto \nabla_X^\mathbb{L} + if.
\]
 In the almost $\S$ case, we show that a suitable example is given by the trivial bundle $M\times\C$. The subset $\mathcal{P}(M,\Phi)$ of the Kostant algebra given by
\[
\mathcal{P}(M,\Phi) = \{(f,X)\in C^\infty(M)\times \Gamma(M,TM) : df = \iota(X)\Phi\},
\]
is a Poisson algebra with respect to the multiplication $(f,X)\cdot (g,Y) = (fg,gX-fY)$ \cite{GGK}. In general there is no canonical notion of a Hamiltonian vector field associated to a function on $M$: any vector field $X$ such that $(f,X)\in \mathcal{P}(M,\Phi)$ is only defined up to sections of $T$,  and not every $f\in C^\infty(M)$ corresponds to a pair in $\mathcal{P}(M,\Phi)$, since if $(f,X)\in\mathcal{P}(M,\Phi)$, then $Yf = 0$ for any section $Y$ of $T$.  In the almost $\S$ case, we show that it is possible to assign a Hamiltonian vector field $X_f$ to each $f\in C^\infty(M)$ such that $\{f,g\} = \eta^i([X_f,X_g])$ defines a Lie bracket on $C^\infty(M)$.  (Here, $\{\eta^i\}$ defines a frame for $(\ker\phi)^*$; on an almost $\S$-manifold it is assumed that $\Phi = -d\eta^i$ for each $i$, whence the bracket does not depend on $i$.)  The resulting vector fields $X_f$ are not symmetries of the almost $\S$-structure unless $(f,X_f)\in\mathcal{P}(M,\Phi)$.

An example of a manifold with almost $\S$-structure is given by a principal $\mathbb{T}^k$-bundle $\pi:(M,\phi)\to (B,\omega)$ over a symplectic manifold $(B,\omega)$ with compatible almost complex structure $J$; we then have $\Phi=\pi^*\omega$ and $\{(f,X_f)\in\mathcal{P}(M,\Phi)\} = \pi^*C^\infty(B,\omega)$.  In this case, one way to think of the quantization of $M$ (at least, for the index-theoretic version) is as a family of quantizations of $(B,\omega)$ parametrized by the finite-dimensional irreducible $\mathbb{T}^k$-representations (consistent with the {\em free action axiom} for the equivariant index \cite{AT,BV2}).

\section{Geometric structures associated to a subbundle}
\subsection{CR and almost CR structures}
In \cite{F2} we concentrated mainly on the case of almost CR structures.  Recall (see \cite{BG} or \cite{DT}, for example) that an {\em almost CR} structure on a manifold $M$ is a constant rank subbundle $E_{1,0}\subset TM\otimes \C$ such that
\begin{equation}\label{aCR}
 E_{1,0}\cap E_{0,1} = 0,
\end{equation}
where $E_{0,1} = \overline{E_{1,0}}$.  An almost CR structure is said to be of type $(n,k)$ if $E_{1,0}$ has complex rank $n$, and $\dim M = 2n+k$.  The rank $2n$ subbundle $E\subset TM$ such that $E\otimes \C = E_{1,0}\oplus E_{0,1}$ is called the {\em Levi distribution} of the almost CR structure.  An almost CR structure is {\em CR-integrable} if the space of sections of $E_{1,0}$ is closed under the Lie bracket, in which case it is simply called a CR structure.

\begin{example}
Many CR manifolds arise as hypersurfaces in complex manifolds.  If $M\subset N$ with $N$ a complex manifold, the subbundle $E_{1,0}\subset T_\C M$ given by $E_{1,0} = T_\C M\cap T^{1,0}N$ defines a CR structure on $M$.

If $(M,E)$ is a contact manifold, then to a contact form $\alpha\in\Gamma(M,E^0\setminus 0)$ (where $E^0\subset T^*M$ denotes the annihilator of $E$) we can associate an almost CR structure on $M$ as follows: since $\alpha$ is a contact form, it follows that $(E,d\alpha)$ is a symplectic vector bundle over $M$, and thus we can choose a fibrewise complex structure on $E$ (that is, $J\in\Gamma(M,\End E)$ with $J^2 = -\Id_E$) that is compatible with the restriction of $d\alpha$ to $E\otimes E$.  Letting $E_{1,0}$ denote the $+i$-eigenbundle of $J$ determines an almost CR structure; if this structure is CR-integrable, then $M$ is a strongly pseudoconvex CR manifold of hypersurface type.  Our results in this particular case can be found in \cite{F1,F4}.
\end{example}

\subsection{The tangential CR complex}\label{crcomp}
While CR structures of type $(n,1)$ (hypersurface type) are the most commonly studied, we will focus in this paper on the case of almost CR structures of type $(n,k)$, for $k>1$, and the consequences of imposing additional conditions (such as integrability) on these structures when necessary.  Of course, our results apply to the cases $k=0$ and $k=1$ as well, but these are already well-served in the literature.

When $M$ is equipped with an (integrable) CR structure $E_{1,0}\subset T_\C M$, it is possible to construct the tangential Cauchy-Riemann complex of $(M,E_{1,0})$. We briefly recall the construction below, and refer the reader to \cite{BG,DT} for more details.  
\begin{remark}
The definition of the tangential CR complex in \cite{DT} is more general, since it works for a CR structure of arbitrary type $(n,k)$, and does not require a choice of any additional structure on $M$; however, it has the disadvantage that elements of the tangential CR complex are not identified with differential forms on $M$.  We will instead follow the approach of \cite{BG}.  The construction in \cite{BG} assumes the existence of a Hermitian inner product on the complexified Levi distribution $E\otimes\C$ such that $E_{1,0}$ and $E_{0,1}$ are orthogonal, and extends (the real part of) this inner product to a Riemannian metric on $TM$ by choosing an orthogonal complement to $E$ in $TM$.  The CR structures we will be dealing with come from $f$-structures, and as explained in Section \ref{fstruct} below,  a CR integrable $f$-structure determines a CR structure together with a choice of complement to the Levi distribution, and it is always possible to choose a metric that is compatible with the CR structure in the sense used in \cite{BG}.
\end{remark}
\begin{remark}
 We will follow the geometric convention that a metric {\em Hermitian} with respect to an almost complex (or almost CR) structure $J$ if $g(JX,JY) = g(X,Y)$ for all appropriate $X,Y$, and continue to use the same letter to denote its $\C$-bilinear extension to the complexification.  This is unfortunately inconsistent with the usual conventions in the theory of complex variables.  To avoid confusion, we will use the notation $\langle Z,W\rangle$ to denote the complex Hermitian metric corresponding to $g$, which is given in terms of $g$ by $\langle Z, W\rangle = g(Z,\overline{W})$, and refer to this as a Hermitian form, or Hermitian inner product, rather than a Hermitian metric. (For the induced inner product on 1-forms, we place the complex conjugation in the first entry.)  Thus, the spaces $E_{1,0}$ and $E_{0,1}$ are orthogonal with respect to the Hermitian inner product $\langle\cdot,\cdot\rangle$, while for the $\C$-bilinear extension of a Hermitian metric $g$, the spaces $E_{1,0}$ and $E_{0,1}$ are isotropic, and $g(Z,\overline{Z})>0$ for any nonzero $Z\in\Gamma(M,E_{1,0})$.
\end{remark}
Let us assume then that $E_{1,0}\subset T_\C M$ is a CR structure on $M$, and that we have chosen a splitting $TM=E\oplus T$, where $E\otimes \C = E_{1,0}\oplus E_{0,1}$.  This splitting gives us the dual splitting $T_\C^*M = E_{1,0}^*\oplus E_{0,1}^*\oplus(T^*\otimes\C)$.  Define $T_{1,0}^*M = E_{1,0}^*\oplus (T^*\otimes\C)$ and $T_{0,1}^*M = E_{0,1}^*$. We can then define the space of $(p,q)$-forms on $M$ by
\begin{equation}\label{CRcomp}
 \mathcal{A}^{p,q}(M,E_{1,0}) = \Gamma(M,\Lambda^p T_{1,0}^*M \wedge \Lambda^q T_{0,1}^*M).
\end{equation}
The space of $l$-forms on $M$ then decomposes according to
\[
 \mathcal{A}^l(M) = \mathcal{A}^{l,0}(M,E_{1,0})\oplus \mathcal{A}^{l-1,1}(M,E_{1,0})\oplus\cdots\oplus\mathcal{A}^{0,l}(M,E_{1,0}),
\]
where some of the above summands may be $\{0\}$ if $l>n$. The tangential $\dbbar$ operator $\dbbar:\mathcal{A}^{p,q}(M,E_{1,0})\to \mathcal{A}^{p,q+1}(M,E_{1,0})$ can then be defined by
\begin{equation}\label{dbbar}
 \dbbar = \pi^{p,q+1}\circ d,
\end{equation}
where $d:\mathcal{A}^{p+q}(M)\to \mathcal{A}^{p+q+1}(M)$ is the usual exterior derivative, and $\pi^{p,q+1}: \mathcal{A}^{p+q+1}(M)\to \mathcal{A}^{p,q+1}(M,E_{1,0})$ is the projection according to the decomposition above.  Using this approach, the operator $\dbbar$ can be defined even for an almost CR structure (this is not true of the definition in \cite{DT}). However, $\dbbar^2=0$ if and only if the almost CR structure is integrable \cite{BG}.  Note that on functions we have $\dbbar f(\overline{Z}) = \overline{Z}\cdot f$.
\begin{definition}
A function $f\in\C^\infty(M,\C)$ is called {\em CR-holomorphic} if it satisfies the tangential CR equations
\begin{equation}\label{tcreq}
\dbbar f = 0.
\end{equation}
\end{definition}
In the case that $M$ is a hypersurface in $\C^n$ the restriction to $M$ of any holomorphic function on $\C^n$ is CR holomorphic.  For a discussion of when the converse is true, see \cite{BG}.

Since $\dbbar^2=0$ on a CR manifold, the above defines a complex on $M$, called the tangential CR complex.  The cohomology of this complex is known as the {\em Kohn-Rossi} cohomology \cite{KR}, which we will denote by $H^{p,q}_{KR}(M,E_{1,0})$.  This complex can also be twisted by a complex vector bundle $\V$, provided that $\V$ is a CR-holomorphic vector bundle in the sense of \cite{DT} (after Tanaka \cite{Tanaka}):
\begin{definition}\label{CRhol}
A {\em CR-holomorphic vector bundle} is a complex vector bundle $\V\to(M,E_{1,0})$ over a CR manifold $(M,E_{1,0})$ equipped with an operator
\[
\overline{\partial}_{\V}:\Gamma(M,\V)\to\Gamma(M,E_{0,1}^*\otimes\V)
\]
such that for any $f\in\C^\infty(M,\C)$, $s\in\Gamma(M,\V)$ and $Z,W\in\Gamma(M,E_{1,0})$,
\begin{enumerate}[(i)]
\item $\overline{\partial}_\V(fs) = f\overline{\partial}_\V s + (\dbbar f)\otimes s$,
\item $[\overline{Z},\overline{W}]s = \overline{Z}\overline{W}s-\overline{W}\overline{Z}s$,
\end{enumerate}
where $\overline{Z}s = \iota(\overline{Z})(\overline{\partial}_\V s)$.
\end{definition}
In other words, in a local trivialization $\V|_U \cong U\times\C^N$, $(\overline{\partial}_\V s|_U)_i = \dbbar s_i$, where $s_1,\ldots, s_N$ are the corresponding components of $s|_U$. The operator $\overline{\partial}_\V$ can be extended to an operator
\begin{equation*}
\overline{\partial}_\V:\A^{0,q}(M,\V)\to \A^{0,q+1}(M,\V)
\end{equation*}
given by
\begin{equation}\label{extend}
 \overline{\partial}_\V(\alpha\otimes s) = (\dbbar \alpha)\otimes s + (-1)^{|\alpha|}\alpha\otimes\overline{\partial}_\V s.
\end{equation}
This operator satisfies $\overline{\partial}_\V^2 = 0$, allowing us to define the twisted Kohn-Rossi cohomology as the cohomology of the resulting complex.

\subsection{$f$-structures}\label{fstruct}
An {\em $f$-structure} on $M$ is an endomorphism field $\varphi\in\Gamma(M,\End TM)$ such that
\begin{equation}\label{fdef}
 \varphi^3+\varphi = 0.
\end{equation}
Such structures were introduced by K. Yano in \cite{Yano}. (See also the text \cite{KY2} for a comprehensive account.)  For the study of $f$-structures in Riemannian geometry we refer to Blair et al \cite{Blair2, BLY}.  By a result of Stong \cite{Stong}, every $f$-structure is necessarily of constant rank.
It is easy to check that the operators $l=-\varphi^2$ and $m=\varphi^2+\Id_{TM}$ are complementary projection operators; letting $E=l(TM) = \Img \varphi$ and $T=m(TM)=\ker \varphi$, we obtain the splitting
\begin{equation}\label{Tsplit}
 TM = E\oplus T = \Img \varphi\oplus \ker \varphi
\end{equation}
of the tangent bundle.  Since $(\varphi|_E)^2 = -\Id_E$, we see that $\varphi$ is necessarily of even rank, and that $\varphi$ determines a splitting $E\otimes\C = E_{1,0}\oplus E_{0,1}$ into the $\pm i$-eigenbundles of $\varphi|_E$.  Thus, as noted in \cite{LP}, an $f$-structure is equivalent to an almost CR structure, together with a choice of complement $T$ to the Levi distribution $E$. 
\begin{example}
 If $\rank \varphi = \dim M$, then $\varphi$ is an almost complex structure on $M$, and $M$ is even dimensional.  If $\rank \varphi = \dim M -1$, then $M$ must be odd dimensional, and the rank one subbundle $T$ must be trivial.  We can then choose a non-vanishing section $\xi$ of $T$, and dual section $\eta$ of $T^*\cong E^0$, such that $(\varphi,\xi,\eta)$ is an almost contact structure on $M$: that is, $\varphi^2 = -\Id + \eta\otimes\xi$, and $\phi(\xi) = \eta\circ\phi=0$.
\end{example}
As noted in the example above, when $\rank T = 1$, $T$ is necessarily trivial.  However, if $\rank T>1$, this need not be the case.  Thus a common (and convenient) additional assumption is that the complement $T$ is trivial, and that a trivializing frame $\{\xi_1,\ldots, \xi_k\}$ has been chosen.  An $f$-structure such that $T$ is trivial is called an {\em $f$-structure with parallelizable kernel} (or $f\cdot$pk-structure) in \cite{LP}.  If we choose a trivializing frame $\{\xi_i\}$ and corresponding coframe $\{\eta^i\}$ for $T^*$, with
\[
 \eta^i(\xi_j) = \delta^i_j,\quad \varphi(\xi_i) = \eta^j\circ\varphi = 0,\quad\text{and}\quad \varphi^2 = -\Id + \sum\eta^i\otimes \xi_i,
\]
then we have what is called an {\em $f$-structure with complemented frames} in \cite{BLY}.  We will at times implicitly assume that an $f\cdot$pk-structure includes a choice of frame and coframe, and adopt the more economical phrase `$f\cdot$pk-structure' in favour of `$f$-structure with complemented frames'.  Given an $f\cdot$pk-structure, it is always possible \cite{KY2} to find a Riemannian metric $g$ that is compatible with $(\varphi,\xi_i,\eta^j)$ in the sense that, for all $X,Y\in \Gamma(M,TM)$, we have
\begin{equation}\label{phig}
 g(X,Y) = g(\varphi X,\varphi Y)+\sum_{i=1}^k \eta^i(X)\eta^i(Y).
\end{equation}
Following \cite{LP}, we will call the 4-tuple $(\varphi,g,\xi_i,\eta^j)$ a {\em metric $f\cdot$pk structure}.  More generally, Soare \cite{Soare} defines a metric $g$ to be compatible with a general $f$-structure $\varphi$ if 
\begin{equation}\label{soare}
 g(\varphi X,Y) + g(X,\varphi Y) = 0,\quad \text{for all}\quad X,Y\in \Gamma(M,TM).
\end{equation}
An simple proof of the existence of a metric $g$ satisfying \eqref{soare} is given in \cite{Soare}, and it is easy to check that any metric satisfying \eqref{phig} also satisfies \eqref{soare}.
\begin{remark}
 If $E=TM$, a metric $f\cdot$pk structure is an almost Hermitian structure, while if $\rank T = 1$, then an $f\cdot$pk-structure is equivalent to an almost contact metric structure.
\end{remark}
Given a compatible pair $(\varphi, g)$, we can define the {\em fundamental 2-form} $\Phi\in\A^2(M)$ by
\begin{equation}\label{2form}
 \Phi(X,Y) = g(\varphi X, Y).
\end{equation}
For later convenience, our fundamental 2-form is the negative of the usual convention found for example in \cite{LP}, which places $\varphi$ in the second slot.  We have adjusted signs accordingly throughout.
An analogue of a contact metric manifold defined in \cite{DIP} is known as an {\em almost $\S$-structure}; this is a metric $f\cdot$pk structure for which $\Phi = -d\eta^i$ for each $i=1,\ldots k$.  An $f$-structure with complemented frames is {\em normal} \cite{BLY,KY2} if
\begin{equation}\label{normalf}
 [\varphi,\varphi] + \sum_{i=1}^k d\eta^i\otimes \xi_i = 0,
\end{equation}
where $[\varphi,\varphi]$ denotes the Nijenhuis tensor of $\varphi$, which is given by
\[
 [\varphi,\varphi](X,Y) = \varphi^2[X,Y]+[\varphi X,\varphi Y]-\varphi[\varphi X,Y]-\varphi[X,\varphi Y].
\]
An almost $\S$-structure that is normal is known as an $\S$-structure. (In \cite{BLY} an $\S$-structure is defined more generally to be a normal $f$-structure such that there exist constants $c_i$ with $\Phi = c_id\eta^i$ for each $i$.  We will assume that each $c_i$ is equal to one.) 
We see that $\Phi$ is antisymmetric from \eqref{soare}, and since $g$ is non-degenerate, it follows that the restriction of $\Phi$ to $E\otimes E$ is also non-degenerate.  Note that if $d\Phi = 0$, then $(E,\Phi|_{E\otimes E})$ is a symplectic vector bundle over $M$. A normal $f$-structure $(\varphi, \xi_i,\eta^j,g)$ with $d\Phi = 0$ is known as a $\mathcal{K}$-structure; in this case the vector fields $\xi^i$ are Killing fields for the metric $\eqref{phig}$ \cite{Blair2}.  In particular, this is true on a manifold with $\S$-structure.  

\subsection{Compatible connections for metric $f\cdot$pk-structures}
From \cite{Soare} we have the result that, given an $f$-structure $\varphi$ with compatible metric $g$, there always exists a connection $\nabla$ on $M$ adapted to the pair $(\varphi,g)$ in the sense that for any $X\in \Gamma(M,TM)$ we have
\begin{equation}\label{soaconn}
 \nabla_X\varphi = \nabla_X g = 0.
\end{equation}
Moreover, from \cite{LP} we have the existence of a canonical connection on a CR-integrable almost $\S$-manifold:
\begin{theorem}\cite{LP}
 Let $M$ be a metric $f\cdot$pk-manifold with structure $(\varphi,\xi_i,\eta^i,g)$.  Then $M$ is a CR-integrable almost $\S$-manifold if and only if there exists a unique linear connection $\nabla$ on $M$ such that
\begin{enumerate}
  \item $\nabla\varphi = \nabla g = \nabla \eta^i = 0$ ($i=1,\ldots, k$)
 \item The torsion $T_\nabla$ of $\nabla$ satisfies \begin{enumerate}
 \item $T_\nabla(X,Y) = -2\Phi(X,Y)\sum \xi_i$ for all $X,Y\in \Gamma(M,E)$
 \item $T_\nabla(\xi_i,\varphi X) = -\varphi T_\nabla(\xi_i,X)$ for all $X\in \Gamma(M,TM)$
 \item $T_\nabla(\xi_i,\xi_j) = 0$ for all $i,j\in \{1,\ldots, k\}$.
\end{enumerate}
\end{enumerate}
\end{theorem}
It follows from the above properties that $\nabla \xi_i = 0$ for all $i=1,\ldots, k$ as well, and that $E$ is parallel with respect to $\nabla$, in the sense that $\nabla_Y X\in \Gamma(M,E)$ for all $X\in \Gamma(M,E)$ and $Y\in\Gamma(M,TM)$.
The above connection is called the {\em generalized Tanaka-Webster connection} in \cite{LP}, as its properties are entirely analogous to those of the Tanaka-Webster connection on a non-degenerate CR manifold of hypersurface type (see \cite{DT}, for example).  Indeed, when $\varphi$ is of type $(n,1)$, the two definitions coincide.

Let $E_{1,0}\subset T_\C M$ denote the CR structure determined by a CR-integrable almost $\S$-structure as above.  As noted in \cite{LP}, CR-integrability is a weaker condition than normality; the almost $\S$-structure $\varphi$ is normal if and only if we in addition have that $T_\nabla(X,\xi_i) = 0$ for all $i=1,\ldots, k$ and all $X\in \Gamma(M,E)$.  

The above conditions can be rephrased in the context of CR geometry.  Included in \cite{LP} is a comparison of their results on the generalized Tanaka-Webster connection with similar results of Mizner \cite{Miz}. From this comparison, we obtain the fact, which we will need later, that property (i) of the torsion $T_\nabla$ above is equivalent to the requirement that for any $Z,W\in \Gamma(M,E_{1,0})$, we have
\begin{equation}\label{torsion}
 T_\nabla(Z,W) = T_\nabla(\overline{Z},\overline{W}) = 0.
\end{equation}

\section{Dirac operators associated to $f$-structures}\label{diracops}
Let us now explain how one can construct a first-order differential operator analogous to the Dolbeault-Dirac operator on any manifold with $f$-structure.  Suppose $M$ is a compact manifold equipped with an $f$-structure $\varphi$ of rank $2n$, and let $TM=E\oplus T$ be the splitting of $TM$ into the image and kernel of $\varphi$.  As noted above, we can then equip $M$ with a Riemannian metric that is compatible with $\varphi$ in the sense that $g(X,\varphi Y) + g(\varphi X,Y) = 0$ for all $X,Y\in \Gamma(M,TM)$.  It follows that $E$ and $T$ are orthogonal with respect to $g$, since if $X=\varphi Y \in \Gamma(M,E)$ and $Z\in \Gamma(M,T)$, then
\[
 g(X,Z) = g(\varphi Y, Z) = -g(Y,\varphi Z) = 0.
\]
Letting $\tilde{g} = g|_E$, and $J = \varphi|_E$ we have $J^2 = -\Id_E$, and $\tilde{g}(JX, JY) = \tilde{g}(X,Y)$ for all $X,Y\in \Gamma(M,E)$.  We then form the ``Clifford bundle'' $\Cl(E)$ whose fibre over $x\in M$ is the complexified Clifford algebra of $E^*_x$ with respect to the bilinear form on $E^*_x$ dual to $\tilde{g}_x$.

Since $J^2 = -\Id_E$, we have the decomposition $E\otimes \C = E_{1,0}\oplus E_{0,1}$ into the $\pm i$-eigenbundles of $J$; as noted above, $E_{1,0}\subset T_\C M$ defines an almost CR structure on $M$.  We can then define the bundle $\S = \Lambda E_{0,1}^*$, which is a Clifford module for $\Cl(E)$ with respect to the action of $\Cl(E)$ on $\S$ defined as follows:  Let $\alpha\in\Gamma(M,E^*)$, and write $\alpha = \alpha^{1,0}+\alpha^{0,1}$ with respect to the splitting of $E^*\otimes \C$ into the $\pm i$-eigenbundles of the complex structure induced on $E^*$ by $J$.  For any $\zeta\in\Gamma(M,\S)$, we set
\begin{equation}\label{clifmult}
 \mathbf{c}(\alpha)\zeta = \sqrt{2}\left(\alpha^{0,1}\wedge\zeta - \iota(\alpha^{1,0})\zeta\right),
\end{equation}
where the contraction is defined using the identification $E_{1,0}^* \cong \overline{E_{1,0}} = E_{0,1}$ determined by the $\C$-bilinear extension of $\tilde{g}$ to $E\otimes\C$.  This is a Clifford action since
\[
 \mathbf{c}(\alpha)^2 = -2\tilde{g}(\alpha^{1,0},\alpha^{0,1}) = -\tilde{g}(\alpha,\alpha).
\]
\begin{remark}
The contraction $\iota(\alpha^{1,0})$ is defined on decomposable elements $\zeta = \beta^1\wedge\cdots\wedge\beta^l$ by 
\[
\iota(\alpha^{1,0})\zeta = \sum(-1)^{i-1}g(\alpha^{1,0},\beta^i)\beta^1\wedge\cdots\wedge\widehat{\beta^i}\wedge\cdots\wedge\beta^l.
\]
If we wished to write this using the Hermitian inner product $\langle\cdot,\cdot\rangle$ instead, we would simply have to note that $g(\alpha^{1,0},\beta^i) = \langle \beta^i,\alpha^{0,1}\rangle$ (since $\alpha$ is real-valued, we have $\alpha^{0,1} = \overline{\alpha^{1,0}}$).  Thus, one often finds the contraction written as $\iota(\alpha^{0,1})$, rather than $\iota(\alpha^{1,0})$, as we do here.
\end{remark}
\begin{remark}\label{BGVClif}
The reader may find it useful to compare our approach to the construction given in \cite{BGV} for the case of a Hermitian manifold $(M,J,g)$.   Using $g$, one constructs the Clifford bundle $\Cl(TM)$, and $\S = \Lambda (T^{0,1}M)^*$ is a spinor module for $\Cl(TM)$ with respect to the Clifford action given by the same formula \eqref{clifmult} as above.  It is known that the Levi-Civita connection $\nabla^{LC}$ associated to the metric $g$ preserves the complex structure $J$ if and only if $(M,J,g)$ is K\"ahler.  In this case $\nabla^{LC}$ preserves the splitting $T_\C M = T^{1,0}M\oplus T^{0,1}M$, and also respects the Clifford multiplication in $\Cl(TM)$.  A connection $\nabla^\V$ on a Clifford module $\V\to M$ is known as a {\em Clifford connection} if it satisfies the following compatibility condition: for all $a\in \Gamma(M,\Cl(TM))$ and all $X\in \Gamma(M,TM)$, we have
\begin{equation}\label{clifconn}
 \left[\nabla^\V_X,\mathbf{c}(a)\right] = \mathbf{c}(\nabla^{LC}_Xa).
\end{equation}
Given a Clifford connection $\nabla^\V$ on a Clifford module $\V\to M$, one can then define a differential operator $\mathsf{D}$ on $\Gamma(M,\V)$ by the composition $\mathsf{D} = \mathbf{c}\circ\nabla^\V$.  An example of a Clifford connection is the connection $\nabla^\S$ given by the connection induced by $\nabla^{LC}$ on $\S$.  Moreover, it is known that any Clifford module $\V$ is locally of the form $\V = \S\otimes \W$ for some complex vector bundle $\W\to M$ with connection $\nabla^\W$, and the connection $\nabla^\V = \nabla^\S\otimes \Id_\W + \Id_\S\otimes \nabla^\W$ is a Clifford connection with respect to the Clifford action $\mathbf{c}(a)\otimes\Id_\W$.  (If $M$ is a spin manifold then this is true globally; this is discussed for example in \cite{Nic}.)
\end{remark}
Let us now return to the case where $M$ is a manifold with $f$-structure $\varphi$ and compatible metric $g$.  From \cite{Soare}, we know that we can find a connection $\nabla$ on $M$ (which unlike $\nabla^{LC}$ will generally have torsion) that preserves both $\varphi$ and $g$.  It follows that $\nabla$ preserves the splitting $E\otimes \C = E_{1,0}\oplus E_{0,1}$ and hence induces a connection $\nabla^\S$ on $\S = \Lambda E_{0,1}^*$.  Since $\nabla$ preserves $g$, it respects the Clifford product on $\Cl(E)$ (that is, $\nabla(ab) = (\nabla a)b+a\nabla b$), and $\nabla^\S$ satisfies a compatibility similar to \eqref{clifconn} above:
\begin{proposition}
 Let $\nabla$ be a connection on $(M,\varphi,g)$ such that $\nabla\varphi = \nabla g = 0$.  Then for any $X\in \Gamma(M,TM)$, the connection $\nabla^\S$ induced by $\nabla$ on $\S$ satisfies $\left[\nabla^\S_X,\mathbf{c}(a)\right] = \mathbf{c}(\nabla_X a).$
\end{proposition}
\begin{proof}
 Since $\nabla$ respects the Clifford multiplication, it suffices to check the result for a 1-form $\alpha\in\Gamma(M,E^*)$. For any $\nu\in\Gamma(M,\S)$, we have
\begin{align*}
 \left[\nabla^\S_X,\mathbf{c}(\alpha)\right]\nu & = \sqrt{2}\nabla^\S_X\left(\alpha^{0,1}\wedge\nu - \iota(\alpha^{1,0})\nu\right) - \sqrt{2}\left(\alpha^{0,1}\wedge \nabla^\S_X\nu - \iota(\alpha^{1,0})\nabla^\S_X\nu\right)\\
& = \sqrt{2}\left(\nabla_X \alpha^{0,1}\wedge\nu + \alpha^{0,1}\wedge\nabla^\S_X\nu - \iota(\nabla_X\alpha^{1,0})\nu\right.\\
& \quad\quad\quad\quad\quad \left.- \iota(\alpha^{1,0})\nabla^\S_X\nu - \alpha^{0,1}\wedge\nabla^\S_X\nu + \iota(\alpha^{1,0})\nabla_X^\S\nu\right)\\
& = \mathbf{c}(\nabla_X\alpha)\nu.\qedhere
\end{align*}
\end{proof}
Using the connection $\nabla^\S$, we can define an odd first-order differential operator 
\[
\mathsf{D}:\Gamma(M,\S^+)\to\Gamma(M,\S^-)
\]
(where $\S^\pm$ denote the subbundles of even and odd forms) by the composition
\begin{equation}\label{Ddef}
 D: \Gamma(M,\S^+)\xrightarrow{\nabla^\S}\Gamma(M,T^*M\otimes\S^+)\xrightarrow{\pi_{E^*}}\Gamma(M,E^*\otimes\S^+)\xrightarrow{\mathbf{c}}\Gamma(M,\S^-),
\end{equation}
where $\pi_{E^*}$ denotes orthogonal projection with respect to $g$.
Motivated by Remark \ref{BGVClif}, we make the following definition:
\begin{definition}
 Let $(M,\varphi,g)$ be a manifold with $f$-structure $\varphi$ and compatible metric $g$, and let $\nabla$ be a metric on $M$ that preserves $\varphi$ and $g$.  Let $\V\to M$ be a vector bundle over $M$ equipped with an action of $\Cl(E)$ and a connection $\nabla^\V$.  We say that the connection $\nabla^\V$ is {\em compatible} with $\nabla$ and the Clifford action if
\begin{equation}
 [\nabla^\V_X,\mathbf{c}(a)] = \mathbf{c}(\nabla_X a)
\end{equation}
for all $a\in\Cl(E)$ and all $X\in\Gamma(M,TM)$.
\end{definition}
For any $\mathbb{Z}_2$-graded Clifford module $\V=\V^+\oplus\V^-$ for $\Cl(E)$, let  $\nabla^\V$ be a compatible connection as defined above. We can then define a differential operator $\mathsf{D}:\Gamma(M,\V^+)\to \Gamma(M,\V^-)$ by the expression \eqref{Ddef} above, by simply replacing $\nabla^\S$ by $\nabla^\V$.  In particular, let $\V = \S\otimes\W$ for some complex vector bundle $\W\to M$ with connection $\nabla^\W$, and equip $\V$ with the product connection $\nabla^\V = \nabla^\S\otimes \Id_\W + \Id_\S\otimes \nabla^\W$.  It's easy to check (the calculation is identical to the one done in \cite{F4}) that $\nabla^\V$ is a compatible connection for the Clifford action $a\mapsto\mathbf{c}(a)\otimes\Id_\W$, allowing us to define the operator $\mathsf{D}_\W:\Gamma(M,\S^+\otimes\W)\to\Gamma(M,\S^-\otimes\W)$ using the connection $\nabla^\V$.
In the case $E=TM$, this definition agrees with the definition of a compatible connection in \cite{Nic}, and thus our operators of the form \eqref{Ddef} can be seen as a generalization of the geometric Dirac operators in \cite{Nic}.  We will thus refer to our operators as ``Dirac'' operators, even though they are not elliptic, except in the case $E=TM$.

\subsection{The CR-integrable almost $\S$ case}\label{3.1}
Let us recall briefly (from \cite{BGV}, for example) that in the case of a K\"ahler manifold, if we take $E=TM$, with $\nabla$ given by the Levi-Civita connection, and if $(\W,\overline{\partial}_\W)\to M$ is a holomorphic vector bundle (where $\overline{\partial}_\W:\mathcal{A}^{p,q}(M,\W)\to\mathcal{A}^{p,q+1}(M,\W)$) equipped with a Hermitian inner product $h$, then there exists a canonical Hermitian connection $\nabla^\W$ on $\W$ such that $\nabla^\W h=0$ and $\nabla^{0,1}:=\nabla^\W|_{T^{0,1}M} = \overline{\partial}_\W$. The Dirac operator $\mathsf{D}_\W$ associated to the tensor product connection on $\Lambda (T^{0,1}M)^*\otimes \W$ is then given by
\begin{equation}\label{holdir}
 \mathsf{D}_\W = \sqrt{2}(\overline{\partial}_\W + \overline{\partial}_\W^*).
\end{equation}
In \cite{F4}, we showed that an analogue of this result holds in the case of a strongly pseudoconvex CR manifold of hypersurface type, if we take $\nabla$ to be the Tanaka-Webster connection of the CR manifold, and take $\W$ to be a CR-holomorphic vector bundle.  Since the generalized Tanaka-Webster connection of \cite{LP} enjoys the same properties as the Tanaka-Webster connection used in the contact case, it's natural to expect that a similar result should hold when our $f$-structure is a CR integrable almost $\S$-structure.

Let us suppose then, that $M$ is an almost $\S$-manifold. Thus, $T=\ker \varphi$ is trivial, and equipped with a frame $\{\xi_i\}$ and corresponding coframe $\{\eta^i\}$ for $T^*$, our metric $g$ can be chosen such that it satisfies \eqref{phig} above, and for each $i=1,\ldots, k$ we have $\Phi = -d\eta^i$, where $\Phi$ is the fundamental 2-form associated to $(\varphi,g)$.  
Let $E\otimes\C = E_{1,0}\oplus E_{0,1}$ be the splitting of $E\otimes \C$ into the $\pm i$-eigenbundles of $\varphi|_E$, and suppose that $E_{1,0}$ defines a CR structure on $M$.  Then we may take $\nabla = \nabla^{LP}$, where $\nabla^{LP}$ denotes the generalized Tanaka-Webster connection of Lotta and Pastore.  Since $\nabla^{LP}\varphi=\nabla^{LP}g = \nabla^{LP}\eta^i = 0$, it follows that $\nabla^{LP}$ preserves the decomposition $T_\C M = E_{1,0}\oplus E_{0,1}\oplus (T\otimes\C)$, and that it respects the Clifford product in $\Cl(E)$.

We now come to one of the main results of this paper: a description of the operator $\mathsf{D}:\Gamma(M,\S^+)\to\Gamma(M,\S^-)$ given by \eqref{Ddef} in terms of the $\dbbar$ operator of the tangential CR complex of $(M,E_{1,0})$.  We are interested in the part of this complex given by
\[
 0\to C^\infty(M)\xrightarrow{\dbbar}\mathcal{A}^{0,1}(M,E_{1,0})\to\cdots\to\mathcal{A}^{0,n}(M,E_{1,0})\to 0.
\]
Using the compatible metric $g$ we define the Hermitian inner product $\langle Z,W\rangle = g(Z,\overline{W})$ on $T_\C M$ with respect to which $E_{1,0}$, $E_{0,1}$, and $T\otimes\C$ are mutually orthogonal.  This induces a pairing $\left<\cdot,\cdot\right>:\mathcal{A}^{i}(M)\times\mathcal{A}^{i}(M)\to C^\infty(M)$ using which we define the inner product
\begin{equation}\label{inprod}
 \left(\psi,\zeta\right) = \int_M\left<\psi,\zeta\right>\mu,
\end{equation}
where $\mu$ is the volume form on $M$ given by 
\begin{equation}\label{mu}
\mu = \eta^1\wedge\cdots\wedge\eta^k\wedge\Phi^n.
\end{equation}
As in \cite{Kohn}, we use the inner product to define the formal adjoint
\[
\dbbar^*:\mathcal{A}^{0,q}(M,E_{1,0})\to\mathcal{A}^{0,q-1}(M,E_{1,0}), 
\]
 given for $\psi\in\mathcal{A}^{0,q}(M,E_{1,0})$ and $\zeta\in\mathcal{A}^{0,q-1}(M,E_{1,0})$ by
\[
 \left(\dbbar^*\psi,\zeta\right) = \left(\psi,\dbbar\zeta\right).
\]
This allows us to construct the operator
\begin{equation}\label{Db}
 \mathsf{D}_b = \sqrt{2}\left(\dbbar+\dbbar^*\right):\Gamma(M,\S)\to\Gamma(M,\S).
\end{equation}
The CR integrability of $E_{1,0}$ implies that $\frac{1}{2}\mathsf{D}_b^2 = \dbbar\dbbar^*+\dbbar^*\dbbar = \Box_b$, the Kohn-Rossi Laplacian \cite{FS,KR}.  Given the action of a group $G$ on $M$ preserving the almost $\S$-structure, the operator $\D_b$ will be $G$-invariant. Letting $[\ker \D_b^+]$ and $[\ker \D_b^-]$ denote the resulting isomorphism classes of $G$-representations, we define the equivariant index of $\D_b$ as the virtual representation
\begin{equation}\label{Geqind}
 \ind^G(\D_b) = [\ker \D_b^+] - [\ker\D_b^-].
\end{equation}
\begin{remark}\label{GGKrem}
 Except in the case that $E=TM$, the virtual representation given by \eqref{Geqind} is infinite-dimensional, and in general it is not clear how to make sense of the above expression; see for example the discussion in \cite[Remark 6.36]{GGK}.  However, as noted there, one can make sense of such expressions for unitary representations in which each finite-dimensional representation occurs with finite multiplicity.  By a result of Atiyah \cite{AT}, this is the case whenever $\D_b$ is transversally elliptic, which is the situation we will consider below.  
\end{remark}
We now come to the main result of this section, which relates the above discussion to our ``geometric Dirac'' operators in the case of an almost $\S$-manifold.
\begin{theorem}\label{main1}
 Let $M$ be a CR-integrable almost $\S$-manifold equipped with the generalized Tanaka-Webster connection $\nabla^{LP}$, and let $\nabla^\S$ be the induced connection on $\S = \Lambda E_{0,1}^*$.  If $\D$ is the operator given by \eqref{Ddef}, then we have the equality
\[
 \D_b = \D.
\]
\end{theorem}
The proof of this result follows from a series of lemmas that we will now proceed to prove, before returning to the proof of the main theorem.
\begin{lemma}\label{lem1}
Let $\{\overline{Z}_1,\ldots,\overline{Z}_n\}$ be a local orthonormal frame for $E_{0,1}$ with respect to the Hermitian pairing $\langle\cdot,\cdot\rangle$, and let $\{\overline{\theta}^1,\ldots, \overline{\theta}^n\}$ denote the corresponding coframe.  Then the tangential $\dbbar$-operator can be expressed in terms of the generalized Tanaka-Webster connection $\nabla^{LP}$ by
\begin{equation}\label{eq1}
\dbbar \gamma= \sum_{i=1}^n\overline{\theta}^i\wedge\nabla^{LP}_{\overline{Z}_i}\gamma,
\end{equation}
for any $\gamma\in\mathcal{A}^{0,q}(M,E_{1,0})$.
\end{lemma}
\begin{proof}
In the case of the Tanaka-Webster connection on a nondegenerate CR manifold of hypersurface type, this result is Proposition 1.17 of \cite{DT}.  Upon inspecting the proof given in \cite{DT}, we see that it relies on two facts: First, that for any $(0,q)$-form $\gamma$, the restriction of $d\gamma$ to $E_{0,1}^{\otimes (q+1)}$ coincides with $\dbbar\gamma$, and second, that the torsion of the Tanaka-Webster connection vanishes on $E_{0,1}\otimes E_{0,1}$.  Since both of these facts remain true for the connection $\nabla^{LP}$ on a CR-integrable almost $\S$-manifold, the proof given in \cite{LP} is equally valid.  The proof is a somewhat lengthy computation, so we do not repeat it here.
\end{proof}
The next lemma can be found in \cite[Appendix 6]{KN1}:
\begin{lemma}\label{lem2}
Let $M$ be an oriented manifold equipped with a volume form $\mu$, and let $\nabla$ be a connection on $M$ such that $\nabla\mu=0$.  Then for any $X\in\Gamma(M,TM)$, the endomorphism $A_X:\Gamma(M,TM)\to \Gamma(M,TM)$ given by $A_X = \mathcal{L}(X)-\nabla_X$ satisfies
\begin{equation}\label{eq2}
\div(X) = -\Tr(A_X),
\end{equation}
where the divergence $\div(X)$ is defined as usual by $\div(X)\mu = \mathcal{L}(X)\mu$.
\end{lemma}
\begin{lemma}\label{lem3}
Let $(M,\phi,\xi_i,\eta^j)$ be a CR-integrable almost $\S$-manifold, equipped with the generalized Tanaka-Webster connection $\nabla^{LP}$ and the volume form $\mu$ given by \eqref{mu}.  Then for any $X\in\Gamma(M,E)$, the endomorphism $\nabla^{LP} X$ given by $\nabla^{LP} X(Y) = \nabla^{LP}_Y X$ satisfies
\begin{equation}\label{eq3}
\Tr(A_X) = -\Tr(\nabla X).
\end{equation}
\end{lemma}
\begin{proof}
We first note that since $\nabla^{LP}\eta^i = \nabla^{LP}g = \nabla^{LP}\phi = 0$, we have $\nabla^{LP}\mu=0$ as well, and thus Lemma \ref{lem2} applies.  For a torsion-free connection, the lemma follows immediately from the identity $A_X(Y)= -\nabla X(Y)-T_\nabla(X,Y)$.  The connection $\nabla^{LP}$, of course, is not torsion-free.  However, we recall that the torsion $\nabla^{LP}$ is explicitly specified by the conditions
\begin{enumerate}[(i)]
\item $T_\nabla(X,Y) = 2\Phi(X,Y)\overline{\xi}$, for all $X,Y\in\Gamma(M,E)$,
\item $T_\nabla(X,\xi_i) = \phi h_i(X)$, for all $i=1,\ldots,k$ and all $X\in \Gamma(M,TM)$,
\item $T_\nabla(\xi_i,\xi_j) = 0$, for all $i,j\in\{1,\ldots, k\}$,
\end{enumerate}
where $\overline{\xi} = \sum_{i=1}^k \xi_i$ and $h_i(X) = (\mathcal{L}(\xi_i)\phi)(X)$. (Condition (ii) above implies the corresponding condition given earlier in the definition of $\nabla^{LP}$.) From \cite{LP}, we know that each operator $h_i$ vanishes on $T=\ker\phi$ and takes values in $\Gamma(M,E)$.  With respect to any local frame $\{X_1,\ldots X_{2n+k}\}$ for $TM$, $\Tr(A_X)$ is given by
\begin{equation}\label{trace}
\Tr(A_X)\mu(X_1,\ldots,X_{2n+k}) = \sum_{i=1}^{2n+k}\mu(X_1,\ldots,A_X(X_i),\ldots,X_{2n+k}).
\end{equation}
For convenience, we choose an orthonormal $\phi$-basis $\{\xi_1,\ldots,\xi_k,e_1,f_1,\ldots,e_n,f_n\}$.  (Recall that the $e_i$ and $f_i$ are a local frame for $E$, and satisfy $f_i = \phi e_i$.) Now, we know that $\mu = \eta^1\wedge\cdots\wedge\eta^k\wedge\Phi^n$, and that $\eta^i(\xi_j)=\delta^i_j$, while $\eta^i(e_j)=\eta^i(f_j)=\iota(\xi_i)\Phi = 0$.  Thus, the only non-zero contributions to the right-hand side of \eqref{trace} must involve $A_X$ in one of three possible ways:
\begin{enumerate}
\item $\eta^i(A_X(\xi_i))$: In this case, we have
\[
A_X(\xi_i) = -\nabla^{LP} X(\xi_i) - T_\nabla(X,\xi_i) = -\nabla^{LP} X(\xi_i)-h_i(X).
\]
But since $h_i$ takes values in $\Gamma(M,E)$, we have $\eta^i(A_X(\xi_i)) = -\eta^i(\nabla^{LP} X(\xi_i))$.
\item $\Phi(A_X(e_i),\cdot)$: Since $X\in\Gamma(M,E)$ and $e_i$ is part of a local frame for $E$, we have
\[
A_X(e_i) = -\nabla^{LP} X(e_i) - T_\nabla(X,e_i) = -\nabla^{LP} X(e_i) -2\Phi(X,e_i)\overline{\xi},
\]
and since $\iota(\xi_i)\Phi = 0$ for $i=1,\ldots, k$, we get $\Phi(A_X(e_i),\cdot) = -\Phi(\nabla^{LP} X(e_i),\cdot)$.
\item $\Phi(A_X(f_i),\cdot)$: In this case we obtain $\Phi(A_X(f_i),\cdot) = -\Phi(\nabla^{LP} X(f_i),\cdot)$ using the same argument as in the previous case.
\end{enumerate}
Thus, we see that for each possibility we may replace $A_X$ on the right-hand side of \eqref{trace} by $-\nabla^{LP} X$, and the result follows.
\end{proof}
\begin{corollary}\label{divtr}
For any $X\in\Gamma(M,E)$, we have $\div(X) = \Tr(\nabla^{LP} X)$, and thus, 
\[
\int_M \Tr(\nabla^{LP} X)\mu = 0.
\]
\end{corollary}
\begin{lemma}\label{lem4}
For any $\alpha\in\mathcal{A}^{0,1}(M,E_{1,0})$, let $X\in\Gamma(M,E_{1,0})$ be the vector field dual to $\alpha$ with respect to the metric $g$; that is, $g(X,Y) = \alpha(Y)$. Let $\{e_i,f_i,\xi_j\}$ be a local $\phi$-basis for $TM$. Then with respect to the local orthonormal frame $\{Z_i\}$ for $E_{1,0}$ given by $Z_j = \frac{1}{\sqrt{2}}(e_j-if_j)$, we have
\begin{equation}\label{eq4}
\Tr(\nabla^{LP} X) = \sum_{i=1}^n\left(Z_i\cdot\alpha(\overline{Z}_i)-\alpha(\nabla^{LP}_{Z_i}\overline{Z}_i)\right).
\end{equation}
\end{lemma}
\begin{proof}
For the local frame $\{Z_i\}$ defined as above, we have $e_i = \frac{1}{\sqrt{2}}(Z_i+\overline{Z}_i)$ and $f_i = \frac{i}{\sqrt{2}}(Z_i-\overline{Z}_i)$.  We can express $\Tr(\nabla^{LP} X)$ using the metric $g$ as
\[
\Tr(\nabla^{LP} X) = \sum_{i=1}^n(g(\nabla^{LP}_{e_i}X,e_i)+g(\nabla^{LP}_{f_i}X,f_i)) + \sum_{i=1}^kg(\nabla^{LP}_{\xi_i}X,\xi_i).
\]
However, $X\in\Gamma(M,E)$ and $\nabla^{LP}$ preserves $E$, and so $g(\nabla^{LP}_{\xi_i}X,\xi_i)=0$, due to the fact that $E$ and $T$ are orthogonal with respect to $g$.  Next, we note that for any $Y_1,Y_2\in\Gamma(M,TM)$, we have
\begin{align*}
(\nabla^{LP}_{Y_1}\alpha)(Y_2) &= Y_1\cdot (\alpha(Y_2))-\alpha(\nabla^{LP}_{Y_1}Y_2)\\
	&= Y_1\cdot g(X,Y_2)-g(X,\nabla^{LP}_{Y_1}Y_2)\\
	&= g(\nabla_{Y_1}X,Y_2).
\end{align*}
Thus we can write $\Tr(\nabla^{LP}X)$ as follows:
\begin{align*}
\Tr(\nabla^{LP}X) &= \sum_{i=1}^n\left(g(\nabla^{LP}_{e_i}X,e_i)+g(\nabla^{LP}_{f_i}X,f_i)\right)\\
	&= \sum_{i=1}^n\left((\nabla_{e_i}\alpha)(e_i)+(\nabla_{f_i}\alpha)(f_i)\right)\\
	&= \sum_{i=1}^n\left(e_i\cdot\alpha(e_i)+f_i\cdot\alpha(f_i)-\alpha(\nabla_{e_i}e_i+\nabla_{f_i}f_i)\right).
\end{align*}
Now, since $\alpha$ is a $(0,1)$-form, we have $\alpha(X) = \alpha(X^{0,1})$.  We have $e_i^{0,1} = \frac{1}{\sqrt{2}}\overline{Z}_i$ and $f_i^{0,1} = -\frac{i}{\sqrt{2}}\overline{Z}_i$, and since $\nabla$ preserves $E_{1,0}$ and $E_{0,1}$, we find that
\[
(\nabla^{LP}_{e_i}e_i)^{0,1} = \frac{1}{2}\left(\nabla^{LP}_{Z_i}\overline{Z}_i + \nabla^{LP}_{\overline{Z}_i}\overline{Z}_i\right),
\]
and
\[
(\nabla^{LP}_{f_i}f_i)^{0,1} = \frac{1}{2}\left(\nabla^{LP}_{Z_i}\overline{Z}_i - \nabla^{LP}_{\overline{Z}_i}\overline{Z}_i\right).
\]
Substituting this into the above expression for $\Tr(\nabla^{LP}X)$, we obtain our result.
\end{proof}
\begin{corollary}\label{barstar}
The formal adjoint $\dbbar^*:\mathcal{A}^{0,q}(M,E_{1,0})\to \mathcal{A}^{0,q-1}(M,E_{1,0})$ is given locally in terms of the connection $\nabla^{LP}$ by
\[
\dbbar^*\gamma = -\sum_{i=1}^n\iota(\overline{Z}_i)\nabla_{Z_i}\gamma.
\]
\end{corollary}
\begin{proof}
Let $\beta\in\mathcal{A}^{0,q}(M,E_{1,0})$ and let $\gamma\in\mathcal{A}^{0,q+1}(M,E_{1,0})$.  Define a 1-form $\alpha\in\mathcal{A}^{0,1}(M,E_{1,0})$ by
\[
\alpha(Y) = \langle\beta,\iota(Y^{0,1})\gamma\rangle,
\]
and let $X\in\Gamma(M,E_{1,0})$ be the dual vector field as in Lemma \ref{lem4}.  Using the result of Lemma \ref{lem4}, we have
\[
\Tr(\nabla^{LP}X) = \sum_{i=1}^n Z_i\cdot\alpha(\overline{Z}_i)-\alpha(\nabla^{LP}_{Z_i}\overline{Z}_i),
\]
and for each $i=1,\ldots, k$, we have
\begin{align*}
Z_i\cdot \alpha(\overline{Z}_i) & = Z_i\cdot \langle \beta,\iota(\overline{Z}_i)\gamma\rangle\\
	&= \langle\nabla^{LP}_{\overline{Z}_i}\beta,\iota(\overline{Z}_i)\gamma\rangle + \langle\beta,\nabla^{LP}_{Z_i}(\iota(\overline{Z}_i)\gamma)\rangle\\
	&= \langle\nabla^{LP}_{\overline{Z}_i}\beta,\iota(\overline{Z}_i)\gamma\rangle + \langle\beta,\iota(\overline{Z}_i)\nabla^{LP}_{Z_i}\gamma\rangle + \alpha(\nabla^{LP}_{Z_i}\overline{Z}_i).
\end{align*}
Since integration over $M$ with respect to the volume form $\mu$ kills $\Tr(\nabla^{LP}X)$ by Corollary \ref{divtr}, we obtain
\begin{equation*}
(\dbbar\beta,\gamma) = \sum_{i=1}^n(\overline{\theta}^i\wedge\nabla^{LP}_{\overline{Z}_i}\beta,\gamma)
	= \sum_{i=1}^n (\nabla^{LP}_{\overline{Z}_i}\beta,\iota(\overline{Z}_i)\gamma)
	= -\sum_{i=1}^n (\beta,\iota(\overline{Z}_i)\nabla^{LP}_{Z_i}\gamma).\qedhere
\end{equation*}
\end{proof}
\begin{proof}[Proof of Theorem \ref{main1}]
Recall that the operator $\D$ is given in this case by the composition $D = \mathbf{c}\circ\pi_{E^*}\circ\nabla^{LP}$, where $\pi_{E^*}:T^*M\to E^*$, and $\mathbf{c}$ denotes the Clifford multiplication.  If we choose a local basis $\{e_i,f_i,\xi_i\}$ adapted to $\phi$, where $f_i=\phi e_i$ and $\phi\xi_i=0$, then we have the local expression
\[
\D = \sum_{i=1}^n(\mathbf{c}(e^i)\nabla^{LP}_{e_i}+\mathbf{c}(f^i)\nabla^{LP}_{f_i})
\]
for $\D$, where $e^i, f^i$ denote the dual forms to $e_i,f_i$, respectively.  In terms of the corresponding basis vectors $Z_j = \frac{1}{\sqrt{2}}(e_j-if_j)$ for $E_{1,0}$ (with $\overline{Z}_j = \frac{1}{\sqrt{2}}(e_j+if_j)$, $\theta^j = \frac{1}{\sqrt{2}}(e^j+if^j)$, and $\overline{\theta}^j = \frac{1}{\sqrt{2}}(e^j-if^j)$ denoting the corresponding basis elements for $E_{0,1}$, $E_{1,0}^*$, and $E_{0,1}^*$), we see that for any $(0,q)$-form $\gamma$, 
\[
\D\gamma = \sqrt{2}\sum_{i=1}^n(\overline{\theta}^i\wedge\nabla^{LP}_{\overline{Z}_i}\gamma-\iota(\theta^i)\nabla^{LP}_{Z_i}\gamma).
\]
Since $\nabla^{LP}$ preserves $\phi$, and hence the splitting $E\otimes\C = E_{1,0}\oplus E_{0,1}$, $\nabla^{LP}_{Z_i}\gamma$ is again a $(0,q)$-form, and thus contraction with $\theta^i$ using the metric $g$ has the same effect as contraction with the vector field $\overline{Z}_i$ in the usual sense, and the result thus follows from Lemma \ref{lem1} and Corollary \ref{barstar}.
\end{proof}
Now let us suppose that $(\V,\overline{\partial}_\V)\to M$ is a CR-holomorphic vector bundle (Definition \ref{CRhol}), equipped with a Hermitian form $h$.  We suppose that $\V$ is equipped with a connection $\nabla^\V$ that is {\em Hermitian} in the sense that, for all $Z\in\Gamma(M,E_{1,0})$, and all $s\in\Gamma(M,\V)$, we have
\begin{equation}\label{hermconn}
\nabla^\V h=0\quad \text{and}\quad \nabla^\V_{\overline{Z}}s = \overline{Z} s := \iota(\overline{Z})\overline{\partial}_\V s.
\end{equation}
\begin{remark}
In the case of a nondegenerate CR manifold of hypersurface type, a ``Hermitian'' connection as described above is unique up to a certain trace condition on its curvature (see \cite{Ura} or \cite[Chapter 8]{DT}).  The case where this trace is zero was introduced in \cite{Tanaka} and is known as Tanaka's canonical connection.
\end{remark}
Given $(\V,h,\nabla^\V)$ as above, we can define the operator $\D_\V$ on sections of $\S\otimes\V$ using the tensor product connection $\nabla=\nabla^{LP}\otimes\Id_\V+\Id_\S\otimes\nabla^\V$, and we have
\begin{theorem}\label{opeqth}
 Let $\overline{\partial}_\V$ continue to denote the extended operator on $\Gamma(M,\S\otimes\V)$ given by \eqref{extend}. We then have the equality of differential operators
\begin{equation}\label{opeq}
 \D_\V = \sqrt{2}\left(\overline{\partial}_\V +\overline{\partial}_\V^*\right):\Gamma(M,\S^+\otimes\V)\to\Gamma(M,\S^-\otimes\V).
\end{equation}
\end{theorem}
\begin{proof}
 We need to show that the two operators agree on sections of $\S\otimes\V$.  The proof is similar to the one given in \cite{F4} for the case of strongly pseudoconvex CR manifolds of hypersurface type, except that in that paper, we were able to make use of a result in \cite{Ura} on the expression of the operator $\overline{\partial}_\nabla^*$ below in terms of the connection $\nabla$.  Let $\{\overline{Z}_i\}$ be a local frame for $E_{0,1}$, with corresponding coframe $\{\overline{\theta}^i\}$ for $E^*_{0,1}$.  We note that $\overline{\partial}_\mathcal{V}$ can be expressed locally by
\begin{equation}\label{dwloc}
 \overline{\partial}_\mathcal{V}(\alpha\otimes s) = \sum \overline{\theta}^i\wedge \overline{Z}_i(\alpha\otimes s),
\end{equation}
where $\overline{Z}_i(\alpha\otimes s) = \iota(\overline{Z}_i)\overline{\partial}_\mathcal{V}(\alpha\otimes s)$.  We define the operator $\overline{\partial}_{\nabla}:\Gamma(M,\Lambda^kE_{0,1}^*\otimes\mathcal{V})\rightarrow \Gamma(M,\Lambda^{k+1} E^*_{0,1}\otimes\mathcal{W})$ given by
\[
 \overline{\partial}_{\nabla} (\alpha\otimes s) = \sum \overline{\theta}^i\wedge (\nabla_{\overline{Z}_i}(\alpha\otimes s)),
\]
for $\alpha\otimes s\in\Gamma(M,\S\otimes\mathcal{V})$.  Then, since $\nabla^\V$ is a Hermitian connection, we have $\nabla_{\overline{Z}}^\V s = \iota(\overline{Z})(\overline{\partial}_\mathcal{V}s)$ for any $\overline{Z}\in \Gamma(M,E_{0,1})$, and therefore,
\begin{align*}
 \overline{\partial}_{\nabla}(\alpha\otimes s) &= \sum \overline{\theta}^i\wedge\left(\nabla^{LP}_{\overline{Z}_i}\alpha\otimes s + \alpha\otimes \nabla^\V_{\overline{Z}_i}s\right)\\
&= \sum\left(\overline{\theta}^i\wedge(\nabla^{LP}_{\overline{Z}_i}\alpha)\otimes s + \overline{\theta}^i\wedge\alpha\otimes (\iota(\overline{Z}_i)\overline{\partial}_\mathcal{V}s)\right)\\
&= (\overline{\partial}_b\alpha)\otimes s + (-1)^{|\alpha|}\alpha\wedge \sum \overline{\theta}^i\otimes \overline{Z}_is\\
&= (\overline{\partial}_b\alpha)\otimes s + (-1)^{|\alpha|}\alpha\wedge (\overline{\partial}_\mathcal{V} s)\\
&= \overline{\partial}_\mathcal{V}(\alpha\otimes s).
\end{align*}
We have thus established that $\overline{\partial}_\V = \overline{\partial}_\nabla$.  It is then easy to check that the proof of Corollary \ref{barstar} applies in this case: we take $\beta\in\mathcal{A}^{0,q}(M,\V)$ and $\gamma\in\mathcal{A}^{0,q+1}(M,\V)$ to be $\V$-valued differential forms, and use the metric $h$ on $\V$ to define a Hermitian pairing on $\V$-valued forms.  We can then define $\alpha(Y) = \langle \beta,\iota(Y^{0,1})\gamma\rangle$ and proceed as before: we have
\[
\Tr(\nabla^{LP}X) = \langle \nabla_{\overline{Z}_i}\beta,\iota(\overline{Z}_i)\gamma\rangle + \langle\beta,\iota(\overline{Z}_i)\nabla_{Z_i}\rangle,
\]
where on the right-hand side $\nabla$ now denotes the tensor product connection, and thus
\[
\overline{\partial}_\V^* = \overline{\partial}_\nabla^* = -\sum_{i=1}^n \iota(\overline{Z}_i)\nabla_{Z_i},
\]
from which the result follows.
\end{proof}

\section{Group actions and transversally elliptic operators}
\subsection{Transverse group actions}
Let $G$ be a compact, connected Lie group acting on a smooth manifold $M$.  Let $E\subset TM$ be a given subbundle.  We say that the action of $G$ on $M$ is {\em transverse} to $E$ if the lifted action of $G$ on $TM$ preserves $E$, and if
\begin{equation}\label{trans1}
 \mathfrak{g}_M+E = TM,
\end{equation}
where $\mathfrak{g}_M$ denotes the space of tangents to the $G$-orbits.  Alternatively, let $\theta\in\mathcal{A}^1(T^*M)$ denote the Liouville 1-form on $T^*M$.  The lifted action of $G$ on $T^*M$ is Hamiltonian with respect to the symplectic form $\omega = -d\theta$, with momentum map $\mu:T^*M\to \mathfrak{g}^*$ given by $\mu^X:=\left<X,\mu\right> = \theta(X_M)$ for all $X\in \mathfrak{g}$, where $X_M$ denotes the vector field on $M$ generated by the infinitesimal action of $X\in\g$.  We can the describe the set of covectors that vanish on $\g_M$ by $T^*_GM = \mu^{-1}(0)$.  Letting $E^0\subset T^*M$ denote the annihilator of $E$, we may rephrase the condition \eqref{trans1} as
\begin{equation}\label{trans2}
 E^0\cap T^*_GM = 0.
\end{equation}
\begin{example}
 Let $P\to B$ be a principal $G$-bundle, and let $HP\subset TP$ be the horizontal bundle with respect to a given choice of connection on $P$.  The vertical bundle $VP$ is identified with $\g_P$, and since $TP=HP\oplus VP$, it follows that the action of $G$ on $P$ is transverse to $HP$.
\end{example}
Given the action of a group $G$ on $M$, transverse to a subbundle $E$, we may, provided that $E$ is cooriented, construct a natural equivariant differential form with generalized coefficients, as described in \cite{F2}.  (An equivariant differential form with generalized coefficients, which we denote by $\alpha(X)\in\mathcal{A}^{-\infty}(M,\mathfrak{g})$, is defined in \cite{KV} to be a $G$-equivariant map from $\g$ to the space of differential forms on $M$ that can be integrated against a $G$-invariant test function on $\g$ to produce a smooth differential form on $M$.  That is, if $\alpha(X)\in\mathcal{A}^{-\infty}(M,\mathfrak{g})$, and $\psi(X)$ is a smooth $G$-invariant function with compact support in $\mathfrak{g}$, then $\int_\mathfrak{g}\alpha(X)\psi(X)dX$ is a differential form on $M$.)  Let $\iota:E^0\hookrightarrow T^*M$ denote the inclusion mapping, and let $p:E^0\to M$ denote projection onto the base.  Using the Liouville 1-form $\theta$ on $T^*M$, one can construct the smooth complex-valued differential form $e^{iD\theta(X)}$, where $D\theta(X) = d\theta-\mu(X)$ is the equivariant symplectic form on $T^*M$.
\begin{remark}
The form $e^{iD\theta(X)}$ appears frequently in symplectic geometry.  Applications considered in \cite{BGV} include equivariant localization, the exact stationary phase approximation of Duistermaat and Heckman \cite{DH}, and the Fourier transform of coadjoint orbits. It also appears in the formula of Berline and Vergne \cite{BV1,BV2} for the equivariant index of transversally elliptic operators.  
\end{remark}
Using the maps $\iota$ and $p$ as defined above, we now define a form
\begin{equation}\label{jex}
 \jex = (2\pi i)^{-k}p_*\iota^*e^{iD\theta(X)},
\end{equation}
for any $X\in\g$, where $k = \rank E^0$.  The assumption that $E$ is cooriented means that the fibres of $E^0$ are oriented, allowing us to define the fibre integration map $p_*$.  The form $\iota^*e^{iD\theta(X)}$ does not have compact support on the fibres of $E^0$, but $\jex$ is defined as an equivariant form with generalized coefficients.  We note that the condition that the group action be transverse to $E$ is necessary, as this ensures that zero is not in the image of the momentum map $\mu_E:E^0\to \g^*$ given by $\left<\mu_E,X\right> = \iota^*\theta(X_M)$.  The differential form \eqref{jex} is treated in detail in \cite{F2} and the author's thesis \cite{F3}, and so we will only outline some of its properties here.  In the case that $E^0$ is trivial, and we are given a trivializing frame $\{\eta^i\}$, let $\boldsymbol{\eta}:E^0\times\g\rightarrow \R^k$ be the map given by
\begin{equation}\label{boldeta}
 \boldsymbol{\eta}_p(X) = (\eta^1_p(X_M),\ldots, \eta^k_p(X_M)),
\end{equation}
and define the equivariant differential of $\boldsymbol{\eta}$ by
\[
 D\boldsymbol{\eta}(X) = (d\eta^1,\ldots,d\eta^k) - \boldsymbol{\eta}(X).
\]
We can then express \eqref{jex} in the more suggestive form
\[
 \jex = \eta^k\wedge\cdots\wedge\eta^1\wedge\delta_0(D\boldsymbol{\eta}(X)),
\]
where $\delta_0$ denotes the Dirac delta on $\R^k$.  While the composition of a distribution on $\R^k$ with an equivariant differential form may seem ill-defined, we can make sense of this expression either as an oscillatory integral (see \cite{Hor}) 
\[
\delta_0(D\boldsymbol{\eta}(X)) = \frac{1}{(2\pi)^k}\int_{(\R^k)^*}e^{-i\left<\zeta,D\boldsymbol{\eta}(X)\right>}d\zeta, 
\]
or via the Taylor expansion
\[
 \delta_0(D\boldsymbol{\eta}(X)) = \sum_{|I|=0}^\infty \frac{\delta_0^{(I)}(\boldsymbol{\eta}(X))}{I!}d\eta^I,
\]
where we use the multi-index notation $I=(i_1,\ldots, i_k)$, $|I|=i_1+\cdots +i_k$, $I! = i_1!\cdots i_k!$, $\delta_0 ^{(I)} = \left(\frac{\partial}{\partial x_1}\right)^{i_1}\cdots\left(\frac{\partial}{\partial x_k}\right)^{i_k}\delta_0$, and $d\eta^I = (d\eta^1)^{i_1}\wedge\cdots\wedge(d\eta^k)^{i_k}$.  (The pullback of $\delta_0^{(I)}$ by $\boldsymbol{\eta}$ to $\mathfrak{g}$ is well-defined by the hypothesis that the action of $G$ is transverse to $E$.)

If $E^0$ is not trivial, we can still make sense of the above expressions locally, and using the scaling properties of the Dirac delta, it's easy to check that the resulting expression for $\jex$ does not depend on the choice of local frame, allowing $\jex$ to be defined globally.  Using the property $x_i\delta_0(\mathbf{x})=0$ of the Dirac delta, it also follows that $\jex$ is equivariantly closed:  $D\jex = 0$.

\subsection{Transversally elliptic operators}
We now recall Atiyah's definition \cite{AT} of  $G$-transversally elliptic operator.  Let $M$ be a compact manifold, and let $\V^\pm\to M$ be two $G$-equivariant vector bundles of rank $l$.  Let $\mathsf{D}:\Gamma(M,\V^+)\to\Gamma(M,\V^-)$ be a pseudodifferential operator, and let 
\[
 \sigma_P(\mathsf{D}):\pi^*\V^+\to \pi^*\V^-
\]
be its principal symbol, where $\pi:T^*M\to M$.
\begin{definition}
 The operator $\mathsf{D}$ is called {\em $G$-transversally elliptic} if
\begin{equation}\label{transell}
 \charr(\sigma_P(\mathsf{D}))\cap T^*_G M = 0,
\end{equation}
where $\charr(\sigma_P(\mathsf{D}))=\{(x,\zeta)\in T^*M : \sigma_P(\mathsf{D})(x,\zeta) \text{ {\rm is not invertible}}\}$.
\end{definition}
We should remark that the above definition is not the most general, but it is the easiest to state, and sufficient for our purposes, since if $G$ acts on $M$ transverse to $E\subset TM$, and $\charr(\sigma_P(\mathsf{D}))\subset E^0$, then $\mathsf{D}$ is a $G$-transversally elliptic operator. In particular, the operators defined in Section \ref{diracops} above have the property that $\sigma_P(\mathsf{D})(x,\zeta) = 0$ for $\zeta\neq 0$ if and only if $\zeta\in E^0$, and thus are transversally elliptic whenever the action of $G$ is transverse to $E$.

Atiyah proved in \cite{AT} that $G$-transversally elliptic operators have a well-defined $G$-equivariant analytic index, given by the virtual character
\[
 \ind^G(\D)(g) = \Tr(g|_{\ker \mathsf{D}}) - \Tr(g|_{\ker \mathsf{D}^*}).
\]
However, unlike in the case of elliptic operators on compact manifolds, the spaces $\ker\mathsf{D}$ and $\ker\mathsf{D}^*$ can be infinite-dimensional, and the equivariant index is defined as a distribution (generalized function) on $G$, rather than as a smooth function.

Berline and Vergne were able to give a cohomological formula for the above index using equivariant differential forms \cite{BV1,BV2}.  Since the index is in general a distribution, it need not be defined pointwise; the advantage of the formula of Berline and Vergne is that it gives the germ of the index near an element $g\in G$.  For elliptic operators on compact manifolds, when the index is smooth, this formula equivalent to the Atiyah-Singer formula \cite{ASing2}, essentially by equivariant localization (see for example \cite[Chapter 8]{BGV}).  In the transversally elliptic case the Duflo-Vergne ``method of descent'' is required \cite{DV2,DV1}. 

Let us briefly recall the formula.  Let $\V=\V^+\oplus\V^-\to M$ be a $G$-equivariant $\Z_2$-graded vector bundle, and suppose $\sigma:\pi^*\V^+\to\pi^*\V^-$ is the symbol of a $G$-transversally elliptic operator (Berline and Vergne deal with what they call ``$G$-transversally good symbols'' - some care has to be taken to ensure $\sigma$ satisfies certain growth conditions on the fibres of $T^*M$).  We assume $\V$ is equipped with a $G$-invariant connection $\nabla=\nabla^+\oplus\nabla^-$ and a $G$-invariant Hermitian metric $h$.  The metric is used to define an endomorphism $v(\sigma)$ of $\pi^*\V$ given by
\[
 v(\sigma) = \begin{pmatrix} 0 & \sigma^*\\ \sigma &0\end{pmatrix},
\]
and using $\nabla$ we define a superconnection
\[
 \mathbb{A}^\theta(\sigma,\nabla,h) = \pi^*\nabla + iv(\sigma) - i\theta\cdot \Id_{\pi^*\V},
\]
where $\theta$ is the Liouville 1-form on $T^*M$.  For any $G$-space $V$, we use the notation $V(g)$ to denote the set of $g$-fixed points. The cohomological index of $\sigma$ is then given near $g\in G$, for sufficiently small $X\in \mathfrak{g}(g)$, by
\[
 \ind^G(\sigma)(ge^X) = \int_{T^*M(g)}(2\pi i)^{-\dim M(g)}\frac{\hatA^2(M(g),X)}{D_g(\mathcal{N},X)}\Ch_g(\mathbb{A}^\theta(\sigma,\nabla,h),X), 
\]
where $\hatA(M(g),X)$ is the equivariant $\hatA$-form, $D_g(\mathcal{N},X)$ is a characteristic form associated to the normal bundle to $M(g)$ in $TM$, and the equivariant Chern character is given by
\[
 \Ch_g(\mathbb{A}^\theta(\sigma, \nabla, h),X) = \Str\left(g\cdot e^{\mathbb{F}(\mathbb{A}^\theta(\sigma, \nabla,h))(X)}\vert_{\V(g)}\right).
\]
Here $\Str$ denotes the supertrace and $\mathbb{F}(\mathbb{A}^\theta(\sigma, \nabla,h))(X)$ is the equivariant curvature of the superconnection $\mathbb{A}^\theta$.  For further details, see \cite{BV1}.
The integrand in this formula is smooth, but since $\mathsf{D}$ may not be elliptic, the Chern character is not compactly supported in general, and the integral is defined in the generalized sense, by its pairing against a test function.  For elliptic symbols the Liouville form $\theta$ can be omitted, and the Chern character becomes an equivariant version Quillen's Chern character \cite{MQ,Q}. 

Later work by Paradan and Vergne \cite{PV1,PV2} showed that it is possible to replace the Chern character in the above formula by a Chern character with compact support, at the expense of allowing this Chern character to be an equivariant differential form with generalized coefficients.  Making this replacement results in the formula for the equivariant index of transversally elliptic operators announced in \cite{PV3}.  In \cite{F3} we showed that the principal symbols of the operators defined in Section \ref{diracops} are the same as the symbol mappings considered in \cite{F2} for the case of almost CR manifolds.  In \cite{F2}, we saw that for such symbols, is possible to integrate over the fibres of $T^*M$ (beginning with either the Berline-Vergne or the Paradan-Vergne formula) to obtain a formula involving the integral of equivariant characteristic forms over the compact manifold $M$.  If $TM=E\oplus T$ denotes the splitting of $TM$ by an $f$-structure $\varphi$, and $\mathsf{D}$ is an operator of the type defined in Section \ref{diracops}, then
\begin{equation}\label{ind1}
 \ind^G(\mathsf{D})(ge^X) = \int_{M(g)}(2\pi i)^{-\rank E(g)/2}\frac{\Td(E(g),X)}{D_g^\C(\mathcal{N}_E,X)}\frac{\hatA^2(T(g),X)}{D_g(\mathcal{N}_T,X)}\mathcal{J}(E(g),X)
\end{equation}
for $X\in \g(g)$ sufficiently small.  In the above formula, $\Td(E(g),X)$ denotes the equivariant Todd form; the terms corresponding to the normal bundle require a brief explanation.  We recall that for any $g\in G$, the restriction of $TM$ to $M(g)$ splits according to $TM|_{M(g)} = TM(g)\oplus\mathcal{N}$, where $\mathcal{N}$ is the normal bundle to $M(g)$.  In general, there is no reason to assume that $\mathcal{N}$ is contained entirely within $E|_{M(g)}$ or $T|_{M(g)}$, so we let $\mathcal{N}_E$ and $\mathcal{N}_T$ denote the respective intersections of these spaces with $\mathcal{N}$.  The fibres of $\mathcal{N}_E$ inherit a complex structure from the almost CR structure; the form $D_g^\C(\mathcal{N}_E,X)$ is defined using the determinant of a complex matrix, rather than the corresponding real matrix of twice the size. (This results in the identity $D_g(\mathcal{N}_E,X) = D_g^\C(\mathcal{N}_E,X)\oplus\overline{D_g^\C(\mathcal{N}_E,X)}$, and the pushforward of the Chern character includes a term that cancels with the complex conjugate.)  Finally, we note that it was proved in \cite{F2} that the action of $G(g)$ on $M(g)$ is transverse to $E(g)\subset TM(g)$, so that $\mathcal{J}(E(g),X)$ remains well-defined.

The advantage of the formula \eqref{ind1} is that it depends only on the given splitting $TM=E\oplus T$ and the group action, and does not involve any concerns such as growth conditions, since the integration is now over the compact fixed-point set $M(g)$.  Furthermore, it has the aesthetic appeal of resembling the Riemann-Roch formula: if $M$ is equipped with a $G$-invariant $f\cdot$pk-structure, then the terms corresponding to $T(g)$ above do not appear.  Moreover, if we consider instead the twisted operator $\D_\W$ acting on sections of $\S\otimes\W$, then we must include the equivariant Chern character of $\W$.  The index formula for the operator $\D_\W$ in the case of an $f\cdot$pk structure is then given near the identity in $G$ by the formula
\[
 \ind^G(\D_\W)(e^X) = \frac{1}{(2\pi i)^{\rank E/2}}\int_M \Td(E,X)\Ch(\W,X)\jex.
\]
\section{Quantization}
We end by explaining how one can produce an analogue of geometric quantization for manifolds with $f$-structure.  In the traditional approach of Kostant \cite{Kost} and Souriau \cite{Sou} in symplectic geometry, we start with a symplectic manifold $M$ with integral symplectic form $\omega$, and construct a complex ``prequantum''  line bundle (or the corresponding circle bundle) whose curvature is equal to $\omega$.   One then defines a Hilbert space given by the space of $L^2$ sections of this bundle, and proceeds to consider observables, polarizations, etc.  If a Lie group $G$ acts on $M$ in a Hamiltonian fashion, this action should then correspond to a representation of $G$ on the Hilbert space of sections.  Details can be found in the text \cite{Wood}.  

Alternatively, it is possible to use a compatible almost complex structure to define a Dirac operator acting on sections of the prequantum line bundle $\mathbb{L}$, and define a graded Hilbert space $Q(M)$ in terms of the kernel and cokernel of this operator.  In the presence of a Hamiltonian group action, this Dirac operator can be defined so that it is $G$-invariant, and one obtains a virtual $G$-representation on $Q(M)$; the character of this representation is then given by the equivariant index of the Dirac operator.  An overview of this approach can be found in \cite{SJ}; Sjamaar calls this approach ``almost complex quantization,'' and points out that we can think of $Q(M)$ as a pushforward $Q(M)=\pi_*([\mathbb{L}])$, where $\pi:M\to\{pt\}$, in equivariant $K$-theory (an idea he attributes to Bott).  Although this can involve physical absurdities such as a negative dimension for $Q(M)$, it has the advantages of requiring less structure on $M$, and allowing the use of tools such as the Atiyah-Singer index theorem.

In the symplectic setting a ``quantization'' should ideally allow us to describe a quantum system entirely in terms of ``classical'' data given by a Hamiltonian action on a symplectic manifold.  In most cases this ideal is overly ambitious, and we generally consider a quantization to be a construction of a Hilbert space $Q(M)$ associated to our manifold $M$, and an assignment $f\mapsto A_f$ of some subset of $C^\infty(M)$ (classical observables) to corresponding self-adjoint (or skew-adjoint, depending on conventions) operators on $Q(M)$ that satisfies the {\em Dirac axioms} (see \cite{GGK}).  

The similarity of the operators defined in Section \ref{diracops} to the Dirac operator used for almost complex quantization leads us to consider the results of the previous section as describing a ``quantization'' of manifolds with $f$-structure: for any $f$-structure $\varphi$, we can choose a compatible metric $g$ and connection $\nabla$, and construct the operator $\D$ defined above.  If a group $G$ acts on $M$ preserving $\varphi$, $g$, and $\nabla$, the operator $\D$ will be $G$-invariant, and the kernel and cokernel of $\D$ will become $G$-representations.  If in addition the action of $G$ is transverse to $E=\varphi(TM)$, then $\D$ will be $G$-transversally elliptic, and the character of the virtual representation $Q(M) = [\ker\D^+] -[\ker\D^-]$ is given by the equivariant index formula of the previous section.  However, in the most general case, there does not seem to be any natural way of choosing a subset of $C^\infty(M)$ to play the role of classical observables, or defining a correspondence with quantum observables.  We can say something if we impose the additional (and typical) condition $d\Phi=0$, where $\Phi$ is the fundamental 2-form associated to the pair $(\phi,g)$, but even then there is nothing canonical.

\subsection{The case $d\Phi=0$}
As we have seen, given an $f$-structure $\varphi$ and compatible metric $g$, we can define the fundamental 2-form $\Phi$, whose restriction to $E$ is non-degenerate.  In order to define an analogue of Kirillov-Kostant quantization for $f$-structures, we must additionally assume that $\Phi$ is closed, so that $(E,\Phi|_{E\otimes E})\to M$ is a symplectic vector bundle.  We can then make use of the notion from \cite{DT}, of a {\em quantum bundle}: we say that a Hermitian line bundle with connection $(\mathbb{L}, h,\nabla)$ is a quantum bundle over $(M,E,\Phi)$ if the restriction of the curvature of $\nabla$ to $E\otimes E$ is equal to $\Phi|_{E\otimes E}$. (If $M$ is symplectic and $E=TM$, we recover the usual definition of a prequantum line bundle.)  We can include the line bundle $\mathbb{L}$ in the differential operator approach described above by making use of the twisted Dirac operator $\D_\mathbb{L}$ acting on sections of $\S\otimes\mathbb{L}$, and defining $Q(M) = [\ker\D_\mathbb{L}^+]-[\ker\D_\mathbb{L}^-]$.

We can also consider an analogue of traditional geometric quantization.  We suppose that $M$ is oriented and equipped with an $f$-structure $\varphi$ and compatible metric $g$, and that the fundamental 2-form $\Phi$ is closed.  We let $(\mathbb{L},h,\nabla)$ be a quantum bundle over $M$.
An inner product on the space of sections of $\mathbb{L}$ is given by
\[
 (s_1,s_2) = \int_M h(s_1,s_2)\mu,
\]
where $\mu$ is a choice of volume form on $M$.  In the case of a metric $f\cdot$pk-structure $(\varphi, g, \xi_i,\eta^i)$, we take $\mu = \eta^1\wedge\cdots \wedge\eta^k\wedge \Phi^n$.  Thus, we may define a prequantization of $(M,\varphi,g)$ to be the space of $L^2$ sections of $\mathbb{L}$ with respect to the above inner product. 

If this structure is CR-integrable, then we have a natural analogue of complex polarization: the CR structure $E_{1,0}\subset T_\C M$ determined by $\varphi$ is integrable, isotropic with respect to the $\C$-linear extension of $g$ to $T_\C M$, and satisfies $E_{1,0}\cap\overline{E_{1,0}} = 0$.  (We drop the maximal rank requirement in favour of the condition $E_{1,0}\oplus E_{0,1} = E\otimes \C$.)  We can then define the quantization of $(M,\varphi,g)$ to be the space of polarized $L^2$ sections of $\mathbb{L}$. Assuming $\mathbb{L}$ is CR-holomorphic and that $\nabla$ is Hermitian, the space of polarized sections of $\mathbb{L}$ is given by the CR-holomorphic sections of $\mathbb{L}$. 

\begin{remark}
When $M$ has a CR-integrable almost $\S$-structure, and $\mathbb{L}$ is CR-holomorphic, we saw in Section \ref{3.1} that $\D_\mathbb{L} = \sqrt{2}(\overline{\partial}_\mathbb{L}+\overline{\partial}_\mathbb{L}^{\, *})$.  It's easy to check that $\ker\D_\mathbb{L} = \ker\D_\mathbb{L}^2 = \{\gamma\in\Gamma(M,\S\otimes\mathbb{L})| \overline{\partial}_\mathbb{L}\gamma=\overline{\partial}_\mathbb{L}^{\,*}\gamma=0\}$. Let us write $\D_\mathbb{L} = \D_\mathbb{L}^+\oplus \D_\mathbb{L}^- :\S^+\otimes\mathbb{L}\oplus \S^-\otimes\mathbb{L}\to \S^-\otimes\mathbb{L}\oplus \S^+\otimes\mathbb{L}$, so that $(\D_\mathbb{L}^+)^* = \D_\mathbb{L}^-$.  Letting $\mathcal{H}^{0,i}_b(M,\mathbb{L})$ denote the space of CR-harmonic $(0,i)$-forms on $M$ with values in $\mathbb{L}$, we can write $\ker\D_\mathbb{L}^+ = \bigoplus^{n/2}_{i=1} \mathcal{H}_b^{0,2i}(M,\mathbb{L})$ and $\ker\D_\mathbb{L}^- = \bigoplus_{i=1}^{n/2} \mathcal{H}^{0,2i-1}_b(M,\mathbb{L})$.
Thus, it is tempting to define $Q(M)$ in terms of the spaces of $\mathbb{L}$-valued CR-harmonic forms of degree $i$, according to
\begin{equation}\label{anind}
 \ind^G(\D_\mathbb{L}) = \sum_{i=0}^n (-1)^i \mathcal{H}^{0,i}_b(M,\mathbb{L}).
\end{equation}
This should be viewed as a formal expression in general; although, as mentioned in Remark \ref{GGKrem}, we can make sense of it as a virtual representation in the case that $\D_\mathbb{L}$ is transversally elliptic.  We may also wish to consider replacing the spaces $\mathcal{H}^{0,i}_b(M,\mathbb{L})$ by some appropriate cohomology groups.  A natural option would be the (twisted) Kohn-Rossi cohomology groups,  but since $Q(M)$ is in general infinite-dimensional,  it may be more appropriate, from the point of view of quantization, to define $Q(M)$ in terms of some $L^2$ cohomology for the $\overline{\partial}_\mathbb{L}$ operator.  (In particular, this would make sense in the almost $\S$ case, where we have a preferred metric and measure; see \cite{GGK}). One advantage of this approach is that the degree zero part of $Q(M)$ then agrees with the definition of $Q(M)$ as the space of CR-holomorphic sections of $\mathbb{L}$.  
\end{remark}

\begin{example}
 Let $(M,\varphi, g, \xi_i,\eta^i)$ be an almost $\S$-manifold.  Let $\mathbb{L} = M\times\C$, equipped with the Hermitian form $h((x,z_1),(x,z_2)) = z_1\overline{z_2}$.  Let $s(x)=(x,f(x))\in\Gamma(M,\mathbb{L})$ denote a section of $\mathbb{L}$, and define a connection on $\mathbb{L}$ by
\begin{equation}\label{lconn}
 \left(\nabla_X s\right)(x) = (x, (X\cdot f)(x) - i\eta(X)f(x)),
\end{equation}
where $\eta(X) = \dfrac{1}{k}\sum \eta^i(X)$.  (We could equally well use any of the individual $\eta^i$, but this choice seems more democratic.)  We then compute that
\[
 \left[\nabla_X,\nabla_Y\right] = [X,Y]-iX\cdot(\eta(Y))+iY\cdot(\eta(X)), 
\]
from which it follows that for $X,Y\in\Gamma(M,E)$,
\[
 \curv\nabla(X,Y) = \left[\nabla_X,\nabla_Y\right]-\nabla_{[X,Y]} = i\Phi(X,Y),
\]
so that $(\mathbb{L},h,\nabla)$ is a quantum bundle over $(M,E,\Phi)$.  Moreover, if we define $(\overline{\partial}_\mathbb{L}s)(x) = (x,(\dbbar f)_x)$, then $\mathbb{L}$ is a CR holomorphic vector bundle, and the connection given by \eqref{lconn} satisfies \eqref{hermconn}.
\end{example}
\subsection{Observables}
 In the above construction of a Hilbert space of sections of a quantum bundle over a manifold with $f$-structure whose fundamental 2-form is closed, we have not provided any discussion of ``observables''.  Since we have a distinguished closed 2-form on $M$, it is possible to construct the Kostant algebra $\mathcal{K}(M,\Phi)$ \cite{Kost,Vais1}.  This is the central extension of the Lie algebra $\Gamma(M,TM)$ given by pairs $(f,X)\in C^\infty(M)\times \Gamma(M,TM)$, together with the bracket
\[
 [(f,X),(g,Y)] = (X\cdot g - Y\cdot f + \Phi(X,Y), [X,Y]).
\]
Given a quantum bundle $(\mathbb{L},h,\nabla)$, we can associate to each $(f,X)\in\mathcal{P}(M,\Phi)$ the skew-Hermitian operator
\begin{equation}\label{Lop}
 A_{(f,x)} = \nabla_X + if 
\end{equation}
on sections of $\mathbb{L}$.  

In general, there is no canonical notion of Hamiltonian vector field for $f$-structures other than those associated to symplectic or contact structures.  In the almost $\S$ case it is possible to define a map $f\mapsto X_f$ from functions to vector fields in a manner after the approach in \cite{Lich1} for contact manifolds, although the resulting vector fields $X_f$ in general will not define infinitesimal symmetries of the almost $\S$-structure.  To do this, we proceed as follows: denote by $\overline{\xi} = \sum \xi_i$ and $\overline{\eta} = \frac{1}{k}\sum\eta^i$.  For any $f\in C^\infty(M)$, we define $X_f$ to be the vector field such that $\iota(X_f)\eta^i = f$ for each $i=1,\ldots, k$, and $\iota(X_f)\Phi = df - (\overline{\xi}\cdot f)\overline{\eta}$. This uniquely defines $X_f$, but it is clearly not canonical.  With this definition we have
\[
 \mathcal{L}(X_f)\eta^i = (\overline{\xi}\cdot f)\overline{\eta}
\]
for each $i$, so that the $X_f$ can't be considered infinitesimal symmetries of the almost $\S$-structure.  However, we can define a bracket on $C^\infty(M)$ by
\[
 \{f,g\} = \iota([X_f,X_g])\overline{\eta} = X_f\cdot g - (\overline{\xi}\cdot f)g;
\]
again, we note that the value of the bracket does not change if we replace $\overline{\eta}$ with any of the $\eta^i$, due to the fact that $\Phi = -d\eta^i$ for each $i$.  Let us confirm that this is in fact a Lie bracket.  Clearly, it suffices to prove the following:
\begin{proposition}\label{hamprop}
 The vector field $X_{\{f,g\}}$ corresponding to the function $\{f,g\}$ is given by $X_{\{f,g\}} = [X_f,X_g]$.
\end{proposition}
\begin{lemma}\label{LL1}
 For each $i=1,\ldots, k$, we have $[\xi_i,X_f] = X_{\xi_i\cdot f}$.
\end{lemma}
\begin{proof}
 From \cite{DIP}, we know that $[\xi_i,\xi_j]=\mathcal{L}(\xi_i)\eta^j = \mathcal{L}(\xi_i)\Phi = 0$ for any $i,j\in\{1,\ldots,k\}$.  Thus, 
\[
\iota([\xi_i,X_f])\eta^i = [\mathcal{L}(\xi_i),\iota(X_f)]\eta^j = \xi_i\cdot f, 
\]
and
\[
 \iota([\xi_i,X_f])\Phi = \mathcal{L}(\xi_i)(df - (\overline{\xi}\cdot f)\overline{\eta}) = d(\xi_i\cdot f) - (\overline{\xi}\cdot(\xi_i\cdot f))\overline{\eta}.\qedhere
\]
\end{proof}
\begin{lemma}\label{LL2}
 For each $i=1,\ldots, k$, we have $\xi_i\cdot\{f,g\} = \{\xi_i\cdot f, g\} + \{f, \xi_i\cdot g\}$.
\end{lemma}
\begin{proof}
 We have, using Lemma \ref{LL1} and the fact that $[\xi_i,\overline{\xi}]=0$ in the second line,
\begin{align*}
\xi_i\cdot\{f,g\} & = \xi_i\cdot (X_f\cdot g) - \xi_i\cdot ((\overline{\eta}\cdot f)g)\\
& =  X_f\cdot (\xi_i\cdot g) - (\overline{\eta}\cdot f)(\xi_i\cdot g) + X_{\xi_i\cdot f}\cdot g - \overline{\eta}(\xi_i\cdot f)g\\
& = \{f,\xi_i\cdot g\} +  \{\xi_i\cdot f,g\}.\qedhere
\end{align*}
\end{proof}
\begin{proof}[Proof of Proposition \ref{hamprop}]
 We already know that $\iota([X_f,X_g])\eta^i = \{f,g\}$ by definition, so it remains to check that
\[
 \iota([X_f,X_g])\Phi = d\{f,g\} - (\overline{\xi}\cdot\{f,g\})\overline{\eta}.
\]
Summing over $i$ in Lemma \ref{LL2}, we have $\overline{\xi}\cdot\{f,g\} = \{f,\overline{\xi}\cdot g\}-\{g,\overline{\xi}\cdot f\} = X_f\cdot(\overline{\xi}\cdot g) - X_g\cdot(\overline{\xi}\cdot f)$.  Thus, we find
\begin{align*}
 \iota([X_f,X_g])\Phi & = \mathcal{L}(X_f)(dg - (\overline{\xi}\cdot g)\overline{\eta}) - \iota(X_g)(-d(\overline{\xi}\cdot f)\wedge\overline{\eta} + (\overline{\xi}\cdot f)\Phi)\\
& = d(X_f\cdot g) - X_f\cdot(\overline{\xi}\cdot g) - (\overline{\xi}\cdot g)(\overline{\xi}\cdot f)\overline{\eta}+ X_g\cdot(\overline{\xi}\cdot f)\overline{\eta}\\
& \quad \quad  - gd(\overline{\xi}\cdot f) - (\overline{\xi}\cdot f)(dg-(\overline{\xi}\cdot g)\overline{\eta})\\
& = d(X_f\cdot g - (\overline{\xi}\cdot f)g) - (X_f\cdot(\overline{\xi}\cdot g) - X_g\cdot(\overline{\xi}\cdot f)\overline{\eta}\\
& = d\{f,g\} - \overline{\xi}\{f,g\}\overline{\eta}.\qedhere
\end{align*}
\end{proof}
We also have the notion from \cite{GGK} of the ``Poisson algebra of a closed 2-form'' this is a Lie subalgebra of $\mathcal{K}(M,\Phi)$ given by
\[
 \mathcal{P}(M,\Phi) = \{(f,X)\in C^\infty(M)\times\Gamma(M,TM) : \iota(X)\Phi = df\},
\]
for the fundamental 2-form $\Phi$.  The bracket of $\mathcal{K}(M,\Phi)$ restricts to the bracket
\[
\{(f,X),(g,Y)\} = (X\cdot g, [X,Y]), 
\]
which is a Poisson bracket with respect to the multiplication $(f,X)\cdot (g,Y) = (fg,gX-fY)$.  Moreover, notice that if $(f,X)\in\mathcal{P}(M,\Phi)$, then $Y\cdot f=0$ for any $Y\in\Gamma(M,\ker\phi)$, since
\[
 Y\cdot f = df(Y) = \Phi(X,Y) = -g(X,\varphi Y) = 0.
\]
In particular, we can consider the pairs $(f,X_f)$, where $X_f$ is the ``Hamiltonian'' vector field corresponding to $f$ defined above, in the case of an almost $\S$-manifold.  Clearly, such a pair belongs to $\mathcal{P}(M,\Phi)$ if and only if $\xi_i\cdot f=0$ for $i=1,\ldots, k$.  By Lemma \ref{LL2}, the set of such pairs forms a subalgebra $\mathcal{P}_b(M,\Phi)\subset \mathcal{P}(M,\Phi)$.  Moreover, if $(f,X_f)\in \mathcal{P}_b(M,\Phi)$, we see that $X_f$ is indeed an infinitesimal symmetry, in the sense that $\mathcal{L}(X_f)(\xi_i) = \mathcal{L}(X_f)\eta^j = \mathcal{L}(X_f)\Phi = 0$ for all $i,j$.
 
\begin{remark}
We note that since $\Phi$ is closed, the distribution $T=\ker\varphi$ is integrable: it's clear that $X\in\Gamma(M,T)$ if and only if $\iota(X)\Phi=0$, and thus, for any $X,Y\in\Gamma(M,T)$ and $Z\in\Gamma(M,TM)$, we have
\[
0=d\Phi(X,Y,Z) = -\Phi([X,Y],Z),
\]
from which it follows that $[X,Y]\in\Gamma(M,T)$.  In \cite{BLY} it was proved that a compact regular $\S$-manifold is a principal torus bundle over a K\"ahler manifold.  (Here, {\em regular} means that the foliation defined by the distribution $T$ is regular, so that the leaf space $M/T$ is a smooth manifold.)  Similarly, a compact regular almost $\S$-manifold is a principal torus bundle over a symplectic manifold.  (A proof of this fact will appear in a forthcoming paper.)  Identifying the $\R^k$-valued form $\boldsymbol{\eta}$ defined by \eqref{boldeta} with a connection form, we can identify $\mathcal{P}(M,\Phi)$ with the horizontal lift of the Poisson algebra associated to the symplectic structure on $M/T$.

When the distribution $T$ is not regular, the algebra $\mathcal{P}(M,\Phi)$ is typically much smaller, and indeed it may be that pairs $(f,X)$ satisfying $df = \iota(X)\Phi$ are only locally defined.  In this case it may be more appropriate to replace $\mathcal{P}(M,\Phi)$ by a suitable space of sheaves, but we have not investigated the usefulness of doing so.
\end{remark}  
\subsection{Stable complex structures and ``symplectization''}
We conclude with a few speculative remarks regarding our construction and some related ideas.  We recall that a stable complex structure on a manifold $M$ is a complex structure defined on the fibres of $TM\oplus \R^k$ for some $k$.  Given an $f\cdot$pk-structure $(\phi,\xi_i,\eta^j)$ on $M$, we obtain a stable complex structure $J\in\Gamma(M,\End(TM\oplus\R^k))$ by setting $JX=\phi X$ for $X\in\Gamma(M,E)$, and defining $J\xi_i = \tau_i$ and $J\tau_i = -\xi_i$, where $\tau_1,\ldots,\tau_k$ is a basis for $\R^k$.  As explained in \cite{GGK}, a stable complex structure determines a Spin$^c$-structure on $M$.  Of course, the resulting Spin$^c$-Dirac operator is elliptic, and the corresponding quantization will be different.

Alternatively, (and with some abuse of notation), we can think of the above complex structure on each fibre $T_xM\times \R^k$ as coming from an almost complex structure on $M\times\R^k$ obtained from to the $f$-structure $\phi$.  We can then define an elliptic Dirac operator associated to this almost complex structure on the non-compact manifold $M\times\R^k$, and this operator, in turn descends to an (again elliptic) operator on $M$.  This approach is used in \cite{Nic} in the case where $M$ is a contact manifold ($k=1$).  We see in \cite{Nic} that in the contact case, the difference between the resulting elliptic operator and the operator $\D$ given by our construction is given by the Lie derivative in the direction of the Reeb vector field.  This suggests that one way to think of our operator is as a deformation of a corresponding elliptic differential operator.  

Finally, we note that almost $\S$-manifolds can be thought of as the higher corank analogues of contact manifolds, and it is possible to define a ``symplectization'' of such manifolds analogous to the symplectization of a cooriented contact manifold.  As above, we let $TM = E\oplus T$ denote the splitting of the tangent bundle determined by the $f$-structure, and let $E^0\cong T^*=\Span\{\eta^i\}\cong M\times \R^k$ denote the annihilator of $E$.  We then have the open  submanifold $E^0_+$ of $E^0$ given by $t_i>0$ for $i=1,\ldots, k$, where $(x,t_1,\ldots,t_k)\in M\times\R^k$. Intrinsically, we can define $E^0_+$ to be the subset of $E^0\setminus 0$ such that $\eta^i(M)\subset E^0_+$ for all $i$, and consider the 2-form $\omega$ given by the pullback to $E^0_+$ of the standard symplectic structure on $T^*M$.  For concreteness, let us use the identification $E^0\cong M\times \R^k$, and with respect to coordinates $(x,t_i,\ldots, t_k)$, let
\[
 \alpha = \sum_{i=1}^k t_i\eta^i,
\]
and define $\omega = -d\alpha$.  Using the fact that $d\eta^i = \Phi$ for each $i$, one can check that
\[
 \omega^{n+k} = \frac{(n+k)!}{n!}\left(\sum_{i=1}^k t_i\right)^ndt_1\wedge \eta^1\wedge\cdots\wedge dt_k\wedge \eta^k\wedge\Phi^n,
\]
and thus the open subset of $M\times\R^k$ defined by $t_i>0$, $i=1,\ldots, k$ is a symplectic submanifold.

\bibliographystyle{alpha}
\bibliography{reference}
\end{document}